\documentclass{compositio}

\usepackage{here}
\usepackage{amsmath, amsthm, amssymb, amsfonts, mathrsfs, ascmac, amscd} 
\usepackage
{xcolor}
\usepackage{enumitem}
\usepackage{mathtools}
\usepackage{comment}
\usepackage{multirow}
\usepackage[all,knot,poly]{xy}
\usepackage{tikz} 
\usepackage{bm}
\usepackage[nameinlink]{cleveref}

\usepackage{url}

\usepackage{longtable}
\usepackage[format=plain, font=normalsize 
]{caption} 

\numberwithin{equation}{section}

\newtheorem{thm}{Theorem}[section]
\newtheorem{lem}[thm]{Lemma}
\newtheorem{cor}[thm]{Corollary}
\newtheorem{prop}[thm]{Proposition}

\theoremstyle{remark}
\newtheorem{ack}{Acknowledgments\!\!}

\theoremstyle{definition}
\newtheorem{dfn}[thm]{Definition}

\newtheorem{eg}[thm]{Example}  
\newtheorem{rem}[thm]{Remark}

\newtheorem{def/prop}[thm]{Definition/Proposition}

\newcommand{\pmx}[1]{\begin{pmatrix}#1\end{pmatrix}}
\newcommand{\spmx}[1]{{\small \pmx{#1}}}

\usepackage[pagewise, mathlines]{lineno} 
\usepackage{time}

\numberwithin{equation}{section}
\usepackage{etoolbox}          

\newcommand*\linenomathpatch[1]{%
  \cspreto{#1}{\linenomath}%
  \cspreto{#1*}{\linenomath}%
  \csappto{end#1}{\endlinenomath}%
  \csappto{end#1*}{\endlinenomath}%
}
\newcommand*\linenomathpatchAMS[1]{%
  \cspreto{#1}{\linenomathAMS}%
  \cspreto{#1*}{\linenomathAMS}%
  \csappto{end#1}{\endlinenomath}%
  \csappto{end#1*}{\endlinenomath}%
}
\expandafter\ifx\linenomath\linenomathWithnumbers 
 \let\linenomathAMS\linenomathWithnumbers
\patchcmd\linenomathAMS{\advance\postdisplaypenalty\linenopenalty}{}{}{}
\else
  \let\linenomathAMS\linenomathNonumbers
\fi

\linenomathpatch{equation}
\linenomathpatchAMS{gather}
\linenomathpatchAMS{multline}
\linenomathpatchAMS{align}
\linenomathpatchAMS{alignat}
\linenomathpatchAMS{flalign}
\makeatletter
\patchcmd{\mmeasure@}{\measuring@true}{
  \measuring@true
  \ifnum-\linenopenaltypar>\interdisplaylinepenalty
    \advance\interdisplaylinepenalty-\linenopenalty
  \fi
  }{}{}
\makeatother

\newcommand{\mf}[1]{{\mathfrak{#1}}}

\newcommand{\bb}[1]{{\mathbb{#1}}}
\newcommand{\mca}[1]{{\mathcal{#1}}}

\newcommand{\inj}{\hookrightarrow}
\newcommand{\surj}{\twoheadrightarrow}
\newcommand{\act}{\curvearrowright}
\newcommand{\congto}{\overset{\cong}{\to}}

\newcommand{\N}{\bb{N}}
\newcommand{\Z}{\bb{Z}}

\newcommand{\Q}{\bb{Q}}

\newcommand{\C}{\bb{C}}

\newcommand{\F}{\bb{F}}

\newcommand{\ccdot}{\!\cdot\!}

\newcommand{\ol}{\overline}
\newcommand{\ul}{\underline}

\newcommand{\ds}{\displaystyle}

\newcommand{\wh}[1]{{\widehat{#1}}}

\renewcommand{\mod}{\ {\rm mod}\ }

\makeatletter
\@namedef{subjclassname@2020}{
\textup{2020} Mathematics Subject Classification}
\makeatother 



\title
[Finite algebraic numbers of positive characteristics]
{{\Large Positive characteristic analogues of finite algebraic numbers}}

\dedication{Dedicated to Professor Masanobu Kaneko}

\author{
Daichi Matsuzuki} 
\email{daichi\_matsuzuki@tamu.edu} 
\address{
Department of Mathematics, Texas A\&M University;
155 Ireland Street, College Station, 77843-3368, TX, USA}

\author{
Honami Sakamoto} 
\email{sakamo10ho73@gmail.com} 
\address{
Graduate School of Humanities and Sciences, Division of Advanced Sciences, 
Department of Mathematics, Ochanomizu University; 2-1-1 Otsuka, Bunkyo-ku, 112-8610, Tokyo, Japan} 

\author{
Jun Ueki
} 
\email{uekijun46@gmail.com}
\address{
Department of Mathematics, Ochanomizu University; 2-1-1 Otsuka, Bunkyo-ku, 112-8610, Tokyo, Japan}

\begin{document}

\begin{abstract} 
J.~Rosen introduced the ring $\mca{P}^0_{\mca{A}}$ of so-called finite algebraic numbers, which may be seen as an analogue of certain periods in the ring $\mca{A}=\prod_p \Z/p\Z /\bigoplus_p \Z/p\Z$, $p$ running through all prime numbers. 
In this article, we introduce its positive characteristic analogue $\mca{P}^0_{\mca{A}_K}$ over the rational function field $K=\F_q(\theta)$, $q$ being a prime power, 
study foundational properties, and provide further scopes. 
\end{abstract}

\maketitle  

{\small 
\tableofcontents
} 

\section{Introduction} 
\subsection{Background}
The ring $\mca{A}=(\prod_p \Z/p\Z)/(\bigoplus_p \Z/p\Z)$, $p$ running through the set of all prime numbers, was introduced by Ax \cite{Ax1968AnnMath} and Kontsevich \cite{Kontsevich2009JJM}, and has played a significant role as the natural habitat for finite multiple zeta values (MZVs) that were initially proposed by Kaneko--Zagier (cf.\,\cite{Kaneko2019PMB, KanekoZagier2026}). 
From the viewpoint of periods of motives, J.~Rosen \cite{Rosen2020JNT} introduced a \emph{finite} analogue $\mca{P}^0_\mca{A}$ of the ring of algebraic numbers within the ring $\mca{A}$, 
and characterized it in terms of linear recurrent sequences.
He also provided notable applications, including an analogue of the Skolem--Mahler--Lech theorem and results concerning the densities of primes (see \Cref{ss.PA0}). 

In the subsequent work, Rosen--Takeyama--Tasaka--Yamamoto \cite{RosenTakeyamaTasakaYamamoto2024JNT} 
discussed a method for encoding the decomposition law of primes into elements of $\mca{P}^0_\mca{A}$, and Hori--Kida \cite{HoriKida2025} further pursued this direction. 
Anzawa--Funakura \cite{AnzawaFunakura2024}, Luca--Zudilin \cite{LucaZudilin2025Ramanujan, LucaZudilin2025-arXiv}, 
Kaneko--Matsusaka--Seki \cite{KanekoMatsusakaSeki2025IMRN}, 
Matsusaka--Miyazaki--Yara \cite{MatsusakaMiyazakiYara2026IJNT}, 
Matsusaka--Seki \cite{MatsusakaSeki2026-arXiv-NaiveTrans}, and Mihara \cite{Mihara2026-arXiv-APNSA}
have conducted stimulating investigations on ``transcendental'' elements in $\mca{A}$, leading the world of numbers to a more profound stratum (see \Cref{eg.trans}). 
Furthermore, there have been studies aiming at bridging the natures of finite and classical MZVs, such as recent noteworthy work by Maesaka--Seki--Watanabe \cite{MaesakaSekiWatanabe2024-arXiv}. 

Now let $K$ be a rational function field in one variable over a finite field. 
The notion of multiple zeta values in positive characteristic was initially introduced by Thakur \cite{Thakur2004}, and intriguing research has flourished in recent years (cf.\,\cite{ChangChenMishiba2022Camb, ChangChenMishiba2023Pi, ImKimLeNgoDacPham2024}). 
The positive characteristic analogue $\mca{A}_K$ of the ring $\mca{A}$, 
introduced by Chang--Mishiba \cite{ChangMishiba2017Bordeaux}, 
serves as the natural habitat for certain finite analogues of multiple zeta values in positive characteristic with recent development (cf.\,\cite{Harada2022JNT, Matsuzuki2022JIS}). 
Accordingly, in view of the venerable recipe of the analogy between number fields and function fields, developments of positive characteristic analogues $\mca{P}^0_{\mca{A}_K}$ of Rosen's ring $\mca{P}^0_{\mca{A}}$, that we will present in this paper (see \Cref{ss.P0AK}), would offer the prospect of a similarly profound research area. 
See \Cref{sec.remarks} for further perspectives. 

\subsection{Rosen's ring $\mca{P}^0_{\mca{A}}$ of finite algebraic numbers}
\label{ss.PA0} 

Consider the ring 
\[\mca{A}=\frac{\prod_p \Z/p\Z}{\bigoplus_p \Z/p\Z},\]
where $p$ runs through the set of all rational primes. 
Let $(a_p)_p,(b_p)_p\in \prod_p \Z/p\Z$ represent $[(a_p)_p], [(b_p)_p]\in \mca{A}$. 
Then we have $[(a_p)_p]=[(b_p)_p]$ iff $a_p=b_p$ holds for every sufficiently large $p$. 
Via the diagonal embedding $\Q\inj \mca{A}$, we often assume $\Q\subset \mca{A}$.

\begin{thm}[({Rosen \cite[Theorem 1.1]{Rosen2020JNT}})] \label{thm.def.A}
Let $\alpha\in \mca{A}$. The following conditions are equivalent. 

{\rm (1)} There is a linear recurrent sequence $(a_n)_n$ over $\Q$, namely, $(a_n)_n\in \Q^\N$ satisfying a linear recurrence relation over $\Q$, such that $\alpha=[(a_p\mod p)_p]$. 

{\rm (2)} There exist a finite Galois extension $L/\Q$ and a map $g:{\rm Gal}(L/\Q)\to L$ satisfying ``$g(\sigma\tau\sigma^{-1})=\sigma(g(\tau))$ for every $\sigma,\tau\in {\rm Gal}(L/\Q)$'' such that $\alpha=[(g(\varphi_p)\mod p)_p]$, where $\varphi_p \in {\rm Gal}(L/\Q)$ denotes ``the Frobenius at $p$'' {\rm (see \Cref{rem.Frob.A})}. 

{\rm (3)} There exists a finite Galois extension $L/\Q$ such that, for an arbitrarily chosen basis of $L$ over $\Q$, $\alpha$ is a $\Q$-linear combination of the matrix coefficients of ``the $\mca{A}$-valued Frobenius automorphism'' $F_\mca{A}:L\otimes \mca{A}\to L\otimes \mca{A}$. 
\end{thm} 

\begin{rem} \label{rem.Frob.A} 
The element $\alpha=[(g(\varphi_p)\mod p)_p]$ in (2) is first defined by choosing a Frobenius element at a prime of $L$ over each $(p)$, but eventually turns out to depend only on $g$ (cf.~{\cite[Section 4]{Rosen2017-arXiv.choice-free}}). 
Namely, let $A(L)$ denote the set of all $g$'s with the condition in (2). Then there is a well-defined map   
${\rm ev}_L:A(L)\to \mca{A}_L;$ $g\mapsto [(g(\varphi_p)\mod p)_p]$
called \emph{the Frobenius-evaluation map}. 

See also Remarks \ref{rem.def(2)} and \ref{rem.AT} for an interpretation of the condition (2). 
\end{rem} 

\begin{dfn}[({\cite[Definition 1.2]{Rosen2020JNT}}+)] 
An element $\alpha\in \mca{A}$ is said to be a \emph{finite algebraic number} if $\alpha$ satisfies the equivalent three conditions in \Cref{thm.def.A}. 
The set of all finite algebraic numbers is denoted by $\mca{P}^0_{\mca{A}}$. 

In addition, the set of all $\alpha\in \mca{A}$ that is algebraic over $\Q$ is denoted by $\mca{C}_\mca{A}$. 
\end{dfn} 

We remark that the name ``finite algebraic number'' is due to subsequent studies of \cite{Rosen2020JNT}. 

\begin{thm}[(Rosen \cite{Rosen2020JNT}, Rosen--Takeyama--Tasaka--Yamamoto 
\cite{RosenTakeyamaTasakaYamamoto2024JNT}, Anzawa--Funakura \cite{AnzawaFunakura2024})] \label{thm.incl} \ 

{\rm (1)} The set $\mca{P}^0_\mca{A}$ is a $\Q$-subalgebra of $\mca{A}$. 

{\rm (2)} There are proper inclusions 
\[\Q\subsetneq \mca{P}^0_{\mca{A}}\subsetneq \mca{C}_{\mca{A}} \subsetneq \mca{A}.\]
\end{thm}

\begin{rem}
\Cref{thm.incl} is proved as follows: 
\renewcommand{\labelitemi}{$\bullet$}
\begin{itemize}
\item The assertion (1) is asserted in \cite{Rosen2020JNT} and explicitly proved in \cite[Proposition 2.7]{RosenTakeyamaTasakaYamamoto2024JNT}. 

\item The proper inclusion $\Q\subsetneq \mca{P}^0_\mca{A}$ follows from \cite[Example 1.5]{Rosen2020JNT} (see \Cref{eg.Fn}) and Dirichlet's density theorem. 

\item The inclusion $\mca{P}^0_\mca{A} \subset \mca{C}_\mca{A}$ is due to the former half of \cite[Theorem 1.4]{Rosen2020JNT}. 

\item The proper inclusion $\mca{P}^0_\mca{A} \subsetneq \mca{C}_\mca{A}$ follows from the facts that $\mca{P}^0_\mca{A}$ is a countable set and $\mca{C}_\mca{A}$ is an uncountable set, as stated in \cite{Rosen2020JNT} and is easily verified (cf. \Cref{lem.P0Csep}). 

\item The proper inclusion $\mca{C}_\mca{A} \subsetneq  \mca{A}$ is proved by \cite[Proposition 3.27, Example 3.8]{AnzawaFunakura2024}.  
\end{itemize} 
See also \Cref{ss.transA} for more about ``transcendental'' elements in several senses. 
\end{rem} 

\begin{eg}[{(Rosen \cite[Example 1.5]{Rosen2020JNT})}] \label{eg.Fn} 
Let $(F_n)_n$ denote the Fibonacci sequence defined by $F_1=F_2=1$ and $F_{n+2}=F_n+F_{n+1}$. Then the Legendre symbol satisfies $F_p\equiv \left(\frac{\,5\,}{p}\right)\mod p$. The Dirichlet density theorem assures that the element $\alpha=[(F_p\mod p)_p]\in \mca{P}^0_{\mca{A}}$ is not in the image of $\Q$ in $\mca{A}$. 
We find that $f(x)=x^2-1$ satisfies $f(\alpha)=0$, so $\alpha\in \mca{C}_\mca{A}$. 
\end{eg}

The classical Skolem--Mahler--Lech theorem asserts the following. 

\begin{thm}[(Skolem \cite{Skolem1934EGDG}, Mahler \cite{Mahler1935ETKF}, Lech \cite{Lech1953AM})]
Let $k$ be a field of characteristic zero and 
let $(a_n)_n$ be a linear recurrent sequence over $k$. 
Then the zero set $\{n\mid a_n=0\}$ is periodic for $n\gg0$, 
that is, this set coincides with the union of finitely many arithmetic progressions 
and a finite set. 
\end{thm} 

To present its finite analogue, we invoke 
\begin{dfn}
For a finite Galois extension $L/\Q$ and a union of conjugacy classes $C\subset {\rm Gal}(L/\Q)$, 
let $S_{L,C}$ denote the set of prime numbers $p$ whose Frobenius conjugacy class is contained in $C$, say, $p$ hitting $C$. 
A set of prime numbers $S$ is said to be \emph{Frobenian} if there is some $(L,C)$ such that $S$ and $S_{L,C}$ coincide except for finitely many elements (cf.~\cite[Section 4.2]{Serre2012NXp}). 
\end{dfn} 
By the Chebotarev density theorem, the natural density of $S_{L,C}$ is $\frac{\#C}{\#{\rm Gal}(L/\Q)}$, so a Frobenian set has rational natural density. 
Now \Cref{thm.def.A} (1) $\Leftrightarrow$ (2) yields a version of the Skolem--Mahler--Lech theorem for the ring $\mca{A}$:

\begin{cor}[(cf.\,Rosen {\cite[Corollary 1.3]{Rosen2020JNT}})] \label{cor.Rosen.1.3} 
 A set $S$ of prime numbers is Frobenian if and only if there is a rational linear recurrent sequence $(a_n)_n$ such that $S$ is a cofinite subset of $\{p\mid a_p\equiv 0\mod p\}$. 
\end{cor}

\Cref{eg.Fn} and \Cref{cor.Rosen.1.3} extend to the study of further constructions of $(a_n)_n$'s in \cite[Theorems 1.3, 1.6]{RosenTakeyamaTasakaYamamoto2024JNT} (see \Cref{rem.S={}}). 

Another consequence of \Cref{thm.def.A} (1) $\Leftrightarrow$ (2) is the following. 

\begin{thm}[(Rosen {\cite[Theorem 1.4, latter half]{Rosen2020JNT}})] \label{thm.Rosen1.4} 
Let $\alpha\in \mca{P}^0_{\mca{A}}$ and let $f(x)\in \Q[x]$ with $f(\alpha)=0$ in $\mca{A}$. 
Then $f(x)$ has a root in $\Q$. 
\end{thm}

Furthermore, we have a formula on the density of $p$'s with $f(a_p)=0$ for $f(x)\in \Q[x]$ and $(a_p)_p\in \prod_p \Z/p\Z$: 

\begin{thm}[(Rosen {\cite[Theorem 1.6]{Rosen2020JNT}})] Let $f(x)\in \Q[x]$. Then the natural densities satisfy 
\[ \underset{[(a_p)_p]\in \mca{P}^0_{\mca{A}}}{\rm sup} \delta(\{p\mid f(a_p)=0\})=
\delta(\{p\mid f(x)\text{\ has a root in\ }\Z/p\Z\}).\] 
If $f(x)$ has no root in $\Q$, then no element of $\mca{P}^0_{\mca{A}}$ realizes the supremum. 
\end{thm} 

\begin{rem} \label{rem.PA0.motive}
As pointed out in \cite[Section 4.1]{Rosen2020JNT}, if we regard ${\rm Spec}\,L$ as a 0-dimensional variety over $\Q$, then the algebraic de Rham cohomology satisfies $H^0_{dR}({\rm Spec}\,L)\cong L$. 
By \Cref{thm.def.A} (3), $\mca{P}^0_{\mca{A}}$ is the $\Q$-span of the matrix coefficients of the $H^0_{dR}({\rm Spec}\,L)\otimes \mca{A} \congto H^0_{dR}({\rm Spec}\,L)\otimes \mca{A}$, where $L/\Q$ runs through all finite Galois extensions. 
If we instead consider the de Rham--Betti comparison isomorphism 
$H^0_{dR}({\rm Spec}\,L)\otimes \C\congto H^0_{B}({\rm Spec}\,L)\otimes \C$, 
then the $\Q$-span of the matrix coefficients becomes an algebraic closure $\ol{\Q}$ of $\Q$. 

Furthermore, let $X$ run through all algebraic varieties over $\Q$ to define 
the $\Q$-subalgebras $\mca{P}_\C\subset \C$ generated by all $\C$-valued periods $\mca{P}_\C$ and $\mca{P}_{\mca{A}}\subset \mca{A}$ generated by all $\mca{A}$-valued periods. 
The parallel inclusions 
\begin{center}
\begin{tabular}{|ccccc|} \hline
$\ol{\Q}$ & $\subsetneq$ & $\mca{P}_\C$ & $\subsetneq$ & $\C$ \\ \hline
$\mca{P}_\mca{A}^0$ & $\subsetneq$ & $\mca{P}_{\mca{A}}$ & $\subsetneq$ & $\mca{A}$ \\ \hline
\end{tabular}
\end{center}
explain that $\mca{P}_\mca{A}^0$, rather than $\mca{C}_\mca{A}$, is a finite analogue of $\ol{\Q}$. 
\end{rem}

\subsection{Positive characteristic analogue $\mca{P}^0_{\mca{A}_K}$ (our results)} \label{ss.P0AK} 

Now let $q$ be a power of a prime number $p$ and let $\theta$ be an indeterminate element. 
We aim to replace $\Z$ and $\Q$ in the previous subsection by 
the polynomial ring $R=\F_q[\theta]$ and 
the rational function field $K={\rm Frac}\,R=\F_q(\theta)$.
Chang--Mishiba's ring is defined by 
\[\mca{A}_K=\frac{\prod_P R/(P)}{\bigoplus_P R/(P)},\]
where $P$ runs through the set of all monic irreducible elements of $R$. 
Via the diagonal embedding $K\inj \mca{A}_K$, we often assume $K\subset \mca{A}_K$. 
 
 \begin{center}
\begin{tabular}{|c||c|} \hline
$\Z$ & $R=\F_q[\theta]$ \\ \hline
$\Q$ & $K=\F_q(\theta)$ \\ \hline
prime number $p$ & monic irreducible $P$ \\ \hline
$\mca{A}$ & $\mca{A}_K$ \\ \hline  
\end{tabular}
\end{center}

We claim that, under the following natural modification, 
the analogues of all assertions in \Cref{ss.PA0} hold true, 
in fact, as Theorems \ref{thm.def.AK}--\ref{thm.density.AK} below demonstrate:

\begin{center}
\begin{tabular}{|c||c|}  \hline
linear recurrence formula & linear recurrence formula with \\ 
& a separable eigen polynomial\\ \hline 
$a_p$ mod $p$ & $a_{q^{{\rm deg}\, P}}$ mod $P$ \\ \hline 
finite extension & finite separable extension\\ \hline 
\end{tabular}
\end{center}

Here, $f(x)\in K[x]$ is said to be \emph{separable} if it has no multiple roots. 
An eigen polynomial is said to be separable (for short) if it is a product of separable polynomials. 

\begin{thm} \label{thm.def.AK} 
Let $\alpha\in \mca{A}_K$. The following conditions are equivalent. 

{\rm (1)} There is a linear recurrent sequence $(a_n)_n$ over $K$ whose eigen polynomial is a product of separable polynomials in $K[x]$ such that $\alpha=[(a_{q^{{\rm deg}P}}\;{\rm mod}\; P)_P]$ holds. 

{\rm (2)} There exist a finite Galois extension $L/K$ and a map $g:{\rm Gal}(L/K)\to L$ satisfying ``$g(\sigma\tau\sigma^{-1}) = \sigma(g(\tau))$ for every $\sigma,\tau\in {\rm Gal}(L/K)$'' such that  $\alpha=[(g(\varphi_P)\;{\rm mod}\; P)_P]$ holds, where $\varphi_P\in {\rm Gal}(L/K)$ denotes ``the Frobenius at $P$'' {\rm (see \Cref{rem.Frobev.AK})}. 

{\rm (3)} There exists a finite Galois extension $L/K$ such that, for an arbitrarily chosen basis of $L$ over $K$, $\alpha$ is a $K$-linear combination of the matrix coefficients of ``the $\mca{A}_K$-valued Frobenius automorphism'' $F_{\mca{A}_K}:L\otimes \mca{A}_K\to L\otimes \mca{A}_K$ {\rm (see \Cref{def.AKFrob})}. 
\end{thm} 

\begin{rem} \label{rem.Frobev.AK}
The element $[(g(\varphi_P)\mod P)_P]$ in \Cref{thm.def.AK} (2) is first defined to be $[(g(\varphi_\mf{p})\mod \mf{p})_\mf{p}]\in \mca{A}_L:=L\otimes \mca{A}_K$, 
but eventually turns out to live in the image of the natural injective homomorphism $\mca{A}_K\inj \mca{A}_L$. 
Namely, let $A(L)$ denote the set of all $g$'s with the condition in (2). 
Then there is a well-defined map 
\[{\rm ev}_L:A(L)\to \mca{A}_K;\ g\mapsto [(g(\varphi_P)\mod P)_P]\]  
called \emph{the Frobenius-evaluation map} (\Cref{prop.Frobev.AK}), 
and the condition (2) becomes ``$\alpha\in {\rm Im}\,{\rm ev}_L$''. 
\end{rem}

\begin{dfn} 
An element $\alpha\in \mca{A}_K$ is said to be a \emph{finite separable element} over $K$ if $\alpha$ satisfies the equivalent three conditions in \Cref{thm.def.AK}. 
The set of all finite separable elements over $K$ is denoted by $\mca{P}^0_{\mca{A}_K}$. 
In addition, the set of all $\alpha\in \mca{A}_K$ that is separable (resp. algebraic) over $K$ is denoted by 
$\mca{C}^{\rm sep}_{\mca{A}_K}$ (resp. $\mca{C}^{\rm alg}_{\mca{A}_K}$). \end{dfn}

\begin{thm} \label{thm.subset.AK} 

{\rm (1)} The set $\mca{P}^0_{\mca{A}_K}$ is a $K$-subalgebra of $\mca{A}_K$. 

{\rm (2)} There are proper inclusions 
\[K\subsetneq \mca{P}^0_{\mca{A}_K} \subsetneq \mca{C}^{{\rm sep}}_{\mca{A}_K} \subsetneq \mca{C}^{{\rm alg}}_{\mca{A}_K} \subsetneq \mca{A}_K.\] 
\end{thm}

\begin{eg} \label{eg.P0AK} 
(1) Suppose $q=2^r$ with $r\in \Z_{>0}$.  
Define a linear recurrent sequence $(F_n)_n\in \F_2^\N\subset R^\N$ by $F_0=F_1=\cdots =F_{q-1}=1$ and $F_{n+q}=F_n+F_{n+1}+\cdots +F_{n+q-1}$, 
so  we have $F_{q}=0$ and $F_{n+q+1}=F_n$. 
By $q^n \equiv (-1)^n \mod (q+1)$, we obtain
$F_{q^n}=F_{(-1)^n}=1+n \mod 2.$ 

(2) Suppose $2\nmid q$. 
Define $(F_n)_n\in R^\N$ by $F_1=1$, $F_2=0$, and $F_{n+2}=\theta F_n$. By $R/(P)^\times \cong \Z/(q^{{\rm deg}P}-1)\Z$, 
the Legendre symbol satisfies $F_{q^{{\rm deg} P}}=\theta^{(q^{{\rm deg}P}-1)/2} \equiv \left(\frac{\theta}{P}\right) \mod P$. 

In both cases, the element $\alpha=[(F_{q^{{\rm deg} P}}\mod P)_P]\in \mca{P}^0_{\mca{A}_K}$ is not in the image of $K$ in $\mca{A}_K$ (see \Cref{lem.KP0}), 
whereas $f(x)=x^2+x$ or $x^2-1$ satisfies $f(\alpha)=0$, so $\alpha\in \mca{C}^{\rm sep}_{\mca{A}_K}$.
\end{eg}

\begin{dfn} \label{def.Frobenian.AK} 
For a finite Galois extension $L/K$ and a union of conjugacy classes $C\subset {\rm Gal}(L/K)$, 
let $S_{L,C}$ denote the set of monic irreducible elements $P$ whose Frobenius conjugacy class is contained in $C$, say, $P$ hitting $C$. 
A set $S$ of monic irreducible elements is said to be \emph{Frobenian} if there is some $(L,C)$ such that $S$ and $S_{L,C}$ coincide except for finitely many elements. 
\end{dfn} 
The Chebotarev density theorem (\Cref{lem.Cheb}) asserts that 
$S_{L,C}$ has the Dirichlet density $\delta(S_{L,C})=\frac{\#C}{\#{\rm Gal}(L/K)}\in \Q$, and so does the Frobenian set $S$. 
Furthermore, if $L/K$ is a geometric extension, then their natural density exists and coincides with their Dirichlet density. 

The following two results are consequences of \Cref{thm.def.AK} (1) $\Leftrightarrow$ (2). 

\begin{cor} \label{cor.Frob.AK} 
A set $S$ of monic irreducible elements of $R$ is Frobenian if and only if 
there is a linear recurrent sequence $(a_n)_n$ over $K$ with separable eigen polynomial
such that $S$ is a cofinite subset of $\{P\mid a_{q^{{\rm deg} P}}\equiv 0\mod P\}$. 
\end{cor}

\begin{thm} \label{thm.root.AK}
Let $\alpha\in \mca{P}^0_{\mca{A}_K}$ and let $f(x)\in K[x]$ with $f(\alpha)=0$ in $\mca{A}_K$. Then $f(x)$ has a root in $K$. 
\end{thm}

Furthermore, we will prove a result on the density of $P$'s with $f(a_{q^{{\rm deg} P}})=0$:  

\begin{thm} \label{thm.density.AK} 
Let $f(x)\in K[x]$ be a product of separable polynomials. Then the densities satisfy 
\[ \underset{[(a_P)_P]\in \mca{P}^0_{\mca{A}_K}}{\rm sup} \delta(\{P\mid f(a_P)=0\})=
\delta(\{P\mid f(x)\text{\ has a root in\ }R/(P)\}).\] 
If $f(x)$ has no root in $K$, then no element of $\mca{P}^0_{\mca{A}_K}$ realizes the supremum. 
\end{thm} 

\begin{rem}\label{rem.t-motive}
Consider a large field $K_\infty=\F_q(\!(1/\theta)\!)$ and let $\C_\infty=\wh{\ol{K_\infty}}$ denote the completion of an algebraic closure. 
In general, a so-called \emph{effective $t$-motive} $M$ that are rigid analytically trivial admits the de Rham--Betti comparison isomorphism 
$H^0_{dR}(M)\otimes \C_\infty\congto H^0_B(M)\otimes \C_\infty$, 
and the presentation matrix yields \emph{the periods} of $M$ of positive characteristic. 
Let $\mca{P}_{\C_\infty}$ denote the $K$-subalgebra of $\C_\infty$ generated by all such periods. 
If $M$ runs through the class of so-called \emph{Artin $t$-motives}, then the $K$-span of the matrix coefficients becomes 
the separable closure $K^{\rm sep}$ of $K$ as a subring of $\mca{P}_{\C_\infty}$ (\Cref{sec.t-motive}, \Cref{prop.ksep}). 
In view of \Cref{rem.PA0.motive}, one may expect to have the following table, explaining that $\mca{P}_{\mca{A}_K}^0$ is a correct finite analogue of $K^{\rm sep}$.
\begin{center}
\begin{tabular}{|ccccc|} \hline
$K^{\rm sep}$ & $\subsetneq$ & $\mca{P}_{\C_\infty}$ & $\subsetneq$ & $\C_\infty$ \\ \hline
$\mca{P}_{\mca{A}_K}^0$ & $\subsetneq$ & $\mca{P}_{\mca{A}_K}$ & $\subsetneq$ & $\mca{A}_K$ \\ \hline
\end{tabular}
\end{center}
Exploring a suitable definition of the finite period ring $\mathcal{P}_{\mathcal{A}_K}$ in positive characteristic thus appears to be an interesting direction for future research. 
\end{rem} 

The rest of this article is organized as follows. 
In \Cref{sec.Frob}, we recall the notion of the Frobenius maps and prove the well-definedness of the Frobenius evaluation map \Cref{prop.Frobev.AK}. 
In \Cref{sec.proof.(1)iff(2)}, we prepare Lemmas \ref{bou} -- 
\ref{lem.diag-perm} to prove \Cref{thm.def.AK} (1) $\Leftrightarrow$ (2). 
In \Cref{sec.proof.(2)iff(3)}, we define the $\mca{A}_K$-valued Frobenius automorphism, prove a lemma, and complete the proof of \Cref{thm.def.AK} (2) $\Leftrightarrow$ (3). 
In \Cref{sec.thm.subset.AK}, we prove \Cref{thm.subset.AK} with several explicit constructions of nontrivial elements and polynomials. We also discuss transcendental elements in $\mca{A}$ and $\mca{A}_K$. 
In \Cref{sec.proofs.Frob/root.AK}, we derive \Cref{cor.Frob.AK} and \Cref{thm.root.AK} from \Cref{thm.def.AK} (1) $\Leftrightarrow$ (2). 
In \Cref{sec.proof.density.AK}, 
we invoke some deep results in algebraic number theory (Lemmas \ref{lem.Cheb}, \ref{lem.principal}),  
recall and refine Rosen's argument (Lemmas \ref{lem.Rosen.1} -- 
\ref{lem.wreathext}), 
and prove \Cref{thm.density.AK} on the densities. 
In \Cref{sec.t-motive}, in view of \Cref{rem.t-motive}, we briefly recall the notion of Artin $t$-motives and prove \Cref{prop.ksep}. 
Further remarks are in \Cref{sec.remarks}. 

\begin{ack}
The authors are grateful to Takumi Anzawa and Hidetaka Funakura for introducing this topic to us, 
Koji Tasaka and Shuji Yamamoto for useful information on technical aspects, 
Yoshinori Mishiba for asking an important question, 
Nadav Gropper and Yi Wang for their insight from anabelian geometry, 
Shin-ichiro Seki and Masanobu Kaneko 
for indicating Ax's paper and sharing the profound mystery of numbers. 
They were also partially assisted by ChatGPT, Grok, and Claude, particularly in locating literature for classical arguments and improving their writing. 
This work was partially supported by RIMS Int.JU/RC, Kyoto University. 
The first and the third authors have been partially supported by 
NSTC, Taiwan (Grant Number NSTC 114-2811-M-007-048) and 
JSPS KAKENHI (Grant Number JP23K12969), respectively. 

\end{ack} 


\section{Frobenius evaluation map} \label{sec.Frob} 
In this section, we recall the notion of Frobenius map and prove \Cref{prop.Frobev.AK}. 

Let the settings be as in \Cref{ss.P0AK}. 
Let $L/K$ be a finite Galois extension with $\Gamma={\rm Gal}(L/K)$ 
and let $R_L$ denote the integral closure of $R=R_K$ in $L$. 
Define an extension of $\mca{A}_K$ by 
$\mca{A}_L
=L\otimes \mca{A}_K
=(\prod_\mf{p} R_L/\mf{p})/(\bigoplus_\mf{p} R_L/\mf{p})$,
where $\mf{p}$ runs through the set of all non-zero prime ideals of $R_L$. 
We may regard $\mca{A}_K\subset \mca{A}_L$ via the natural injective homomorphism. 

\begin{dfn}[({Frobenius map \cite[Section 9]{Rosen2002GTM}})] 
Let $P$ be a monic irreducible element of $R$ and 
let $\mf{p}$ be a prime ideal of $R_L$ over $(P)$. 
Then an element $\varphi_\mf{p}\in \Gamma$ satisfying 
$\varphi_\mf{p}(x)\equiv x^{\# R/(P)}\mod \mf{p}$
is called a Frobenius map at $\mf{p}$. 
\end{dfn} 

If $\mf{p}$ is unramified in $L/K$, then $\varphi_\mf{p}$ is unique. 
Otherwise, there are several Frobenius maps, all conjugate to each other. 
Since there are only finitely many ramified primes, we have 

\begin{lem} \label{lem.Frobmfp}
For each map $g:\Gamma\to L$, 
the element $[(g(\varphi_\mf{p}) \mod \mf{p})_\mf{p}]\in \mca{A}_L$ is defined. 
\end{lem} 

\begin{lem} \label{lem.FrobP}
Let $L/K$ be a finite Galois extension. 
Let $P$ be a monic irreducible element of $R$ and 
let $\mf{p},\mf{p}'$ be unramified prime ideals of $R_L$ over $(P)$. 
Let $g:\Gamma\to L$ be in $A(L)$. 
Regard $R/(P)\subset R_L/\mf{p}$ and $R/(P)\subset R_L/\mf{p}'$ via the natural injective homomorphisms. 
Then, we have 

{\rm (1)} $g(\varphi_\mf{p})\mod \mf{p} \in R/(P)$, 

{\rm (2)} $g(\varphi_\mf{p})\mod \mf{p}=g(\varphi_\mf{p'})\mod \mf{p}'$ in $R/(P)$. 
\end{lem} 

\begin{proof} 
(1) By $g \in A(L)$, we have $\sigma(g(\tau))=g(\sigma \tau \sigma^{-1})$ for every $\sigma, \tau \in \Gamma$. Putting $\sigma=\tau=\varphi_{\mathfrak{p}}$, we obtain $\varphi_{\mathfrak{p}}(g(\varphi_{\mathfrak{p}}))=g(\varphi_{\mathfrak{p}})$. 
By the definition of $\varphi_\mf{p}$, we have $\mf{p}^{\varphi_\mf{p}}=\mf{p}$, so it acts on $R_L/\mf{p}$. Let $\ol{\varphi}_\mf{p}\in {\rm Gal}((R_L/\mf{p})/(R/(P)))$ denote the corresponding element. 
Again by the definition of $\varphi_\mf{p}$, we have 
\[g(\varphi_{\mathfrak{p}}) \mod \mathfrak{p}=\ol{\varphi}_{\mathfrak{p}}(g(\varphi_{\mathfrak{p}}) \mod \mathfrak{p})=(g(\varphi_{\mathfrak{p}})\mod \mathfrak{p})^{\#(R/(P))}\]
in $R/\mf{p}$. 
Recall the natural injective homomorphism  $R/(P)\congto (R_L/\mf{p})^{\langle \ol{\varphi}_\mf{p}\rangle} \inj R_L/\mf{p}$. 
If $a\in R$ $b\in R_L$, then we have $a\mod (P)\mapsto b\mod \mf{p}$ iff $b-a \in \mf{p}$. 
Thus we have $g(\varphi_{\mathfrak{p}})\mod \mf{p} \in R/(P)$. 

(2) Since $L/K$ is Galois, we have $\sigma\in \Gamma$ with 
$\sigma\mf{p}=\mf{p}'$, so 
$\varphi_{\mf{p}'}=\sigma\varphi_\mf{p}\sigma^{-1}$. 
By $g\in A(L)$, we have 
$g(\varphi_{\mf{p}'})=g(\sigma\varphi_\mf{p}\sigma^{-1})
=\sigma(g(\varphi_\mf{p})) \in R_L$. 
By (1), we have $a\in R$ with $g(\varphi_{\mathfrak{p}}) -a \in \mf{p}$, and hence 
$\sigma(g(\varphi_{\mathfrak{p}}) -a) = g(\varphi_{\mf{p}'})-a \in \mf{p}'$. 
Namely, $g(\varphi_{\mf{p}})\mod \mf{p}$ and $g(\varphi_{\mf{p}'})\mod \mf{p}'$ coincide with the same element $a\mod (P)$ in $R/(P)$. 
\end{proof}

\begin{prop} \label{prop.Frobev.AK}
For a finite Galois extension $L/K$, 
there is a well-defined map   
\[{\rm ev}_L:A(L)\to \mca{A}_K;\ g\mapsto [(g(\varphi_P)\mod P)_P].\] 
\end{prop} 

\begin{proof} 
By \Cref{lem.FrobP}, the element $(g(\varphi_\mf{p}) \mod \mf{p})_{\mf{p}\mid (P)}\in \prod_{\mf{p}\mid (P)}R_L/\mf{p}$ lives in the image of $R/(P)$. 
So we may let $g(\varphi_P) \mod P$ denote the corresponding element in $R/(P)$ and call it the value of the Frobenius at $P$.
Accordingly, the element $[(g(\varphi_\mf{p}) \mod \mf{p})_\mf{p}]\in \mca{A}_L$ in \Cref{lem.Frobmfp} lives in the image of $\mca{A}_K$, so we may write $[(g(\varphi_P) \mod P)_P]\in \mca{A}_K$. 
\end{proof} 

\section{Proof of \Cref{thm.def.AK} (1) $\Leftrightarrow$ (2)}  \label{sec.proof.(1)iff(2)} 


\subsection{Preliminaries}

Let the setting be as in \Cref{ss.P0AK}.
Let $L/K$ be a finite Galois extension with $\Gamma={\rm Gal}(L/K)$ and put $d=[L:K]=\#\Gamma$. 
Let $L\otimes L$ denote the tensor product over $K$. 
The following deep assertion will play a key role.
\begin{lem}[(cf. Bourbaki, {\cite[Chapter V, Section 10, 4, Corollary of Proposition 8]{Bourbaki.AlbegraII}})]\label{bou}
The canonical $K$-linear map 
\[\psi: L\otimes L\to {\rm Map}(\Gamma, L);\ \sum_i x_i\otimes y_i \mapsto (\tau\mapsto \sum_i x_i\tau(y_i))\] 
is bijective. 
\end{lem} 

\begin{lem} \label{change} \label{lem.GinvAL}
Via Bourbaki's canonical isomorphism $\psi:L\otimes L \congto {\rm Map}(\Gamma,L)$, 
the $\Gamma$-invariant subset $(L\otimes_K L)^\Gamma$ with respect to the diagonal action corresponds to the subset 
\[A(L)=\{g\in{\rm Map}(\Gamma,L)\mid g(\sigma\tau\sigma^{-1})=\sigma g(\tau)\ \text{for every}\ \sigma,\tau\in \Gamma\}.\]
\end{lem} 

\begin{proof} Let $a=\sum_i x_i\otimes y_i \in L\otimes L$. 
Then for every $\sigma,\tau\in \Gamma$, we have 
\begin{align*}
&\psi(a)(\sigma\tau\sigma^{-1})
=\psi(\sum_i x_i\otimes y_i)(\sigma\tau\sigma^{-1})
=\sum_i x_i\cdot (\sigma\tau\sigma^{-1})y_i\\ 
&=\sigma \sum_i (\sigma^{-1}x_i)\cdot \tau(\sigma^{-1}y_i)
=\sigma \psi(\sigma^{-1}(\sum_i x_i\otimes y_i))(\tau)
=\sigma \psi(\sigma^{-1}(a))(\tau).
\end{align*}

If $a\in (L\otimes_K L)^\Gamma$, then 
$\psi(a)(\sigma\tau\sigma^{-1})=\sigma \psi(a)(\tau)$, namely, $\psi(a)\in A(L)$. 

Conversely, if $\psi(a)\in A(L)$, then 
$\sigma \psi(a)(\tau)= \psi(a)(\sigma\tau\sigma^{-1})=\sigma \psi(\sigma^{-1}a)(\tau)$ 
implies $a=\sigma^{-1}a$ for every $\sigma\in \Gamma$, so $a\in (L\otimes L)^\Gamma$.
\end{proof}

The normal basis theorem yields 
\begin{lem} \label{tikan} \label{lem.Ginv-perm} 
There is a surjective homomorphism 
\[{\rm Tr}_{L/K}:L\otimes L\surj (L\otimes L)^\Gamma;\ \sum_\iota x_\iota\otimes y_\iota\mapsto \sum_{\sigma\in \Gamma} \sum_\iota  \sigma(x_\iota)\otimes \sigma(y_\iota).\] 
Namely, if $a\in (L\otimes L)^\Gamma$, then 
$a=\sum_i x_i\otimes y_i$ for some $(x_i)_i,(y_i)_i\in L^{d^2}$ such that 
$\Gamma$ acts on their indices in the same way. 
\end{lem}

\begin{proof} 
Take a normal $K$-basis $(\beta_j)_{j\in J}$ of $L$, so 
$\sigma \in \Gamma$ regularly acts on the index set $J=\{1,2,\ldots, d\}$ via $\sigma(\beta_j)=\beta_{\sigma(j)}$. 
Take a $K$-basis $(\gamma_i)_{i\in I}$ of $L\otimes L$ with $I=\{1,2,\ldots, d^2\}$ such that $\{\gamma_i\mid i\in I\}=\{\beta_j\otimes \beta_k\mid j,k\in J\}$.  
Then $\Gamma$ freely acts on $I$ via $\sigma(\gamma_i)=\gamma_{\sigma(i)}$. 
Write $a=\sum_i s_i \gamma_i \in (L\otimes_K L)^\Gamma$ 
with $s_i\in K$. 
Then, for every $\sigma\in \Gamma$, we have 
$\sum_i s_i \gamma_i=a=\sigma(a)=\sum_i s_i \sigma(\gamma_i)=\sum_i s_i \gamma_{\sigma(i)}=\sum_i s_{\sigma^{-1}(i)} \gamma_i$, hence 
$s_i=s_{\sigma^{-1}(i)}$ for every $i$. 
Let $\iota$ run through a complete system of representatives of the family $I/\Gamma$ of $\Gamma$-orbits of $i$'s. 
Then we have $a=\sum_\iota s_\iota\sum_{\sigma\in \Gamma} \gamma_{\sigma(\iota)}$. 

For each $i\in I$, define $j_i,k_i\in J$ by $\gamma_i=\beta_{j_i}\otimes \beta_{k_i}$. 
If $i\in \Gamma (\iota)$ with $s_\iota\neq 0$, put $(x_i,y_i)=(s_\iota \beta_{j_i},\beta_{k_i})$. 
Otherwise, put $(x_i,y_i)=(0,0)$. 
Then $a=\sum_i x_i\otimes y_i$. 
Moreover, for each $i\in \Gamma(\iota)$, we have 
$\sigma(x_i)\otimes \sigma(y_i)
=\sigma(x_i\otimes y_i)
=\sigma(s_\iota \gamma_i)
=s_\iota \gamma_{\sigma(i)}
=x_{\sigma(i)}\otimes y_{\sigma(i)}$. 
Since $\Gamma$ freely acts on non-zero $x_i\otimes y_i$'s, we obtain  
$\sigma(x_i)=x_{\sigma(i)}$ and $\sigma(y_i)=y_{\sigma(i)}$ for every $i\in I$.  
Namely, 
$\Gamma$ acts on the indices of $(x_i)_i$ and $(y_i)_i$ in the same way. 
\end{proof}

A version of Hilbert's Satz 90 yields the following. 

\begin{lem} \label{change} \label{lem.diag-perm} 
Let $S\in {\rm GL}_l(K)$ be diagonalizable over $L$ with the (not necessarily distinct) eigenvalues $\lambda_1,\ldots, \lambda_l \in L$ 
and let $\bm{u}=(u_i)_i, \bm{v}=(v_i)_i\in K^l$. 
Then there exist $\bm{b}=(b_i)_i \in L^l$ such that 
\[\bm{u}^{\!\top}S^n\bm{v}=\sum_i b_i \lambda_i^n\]
for every $n\in \Z$ 
and that $\Gamma={\rm Gal}(L/K)$ acts on the set $\{(b_i,\lambda_i)\mid i\}$ 
of pairs. 
\end{lem}

\begin{proof} 
Put $D=\pmx{\lambda_1&&\\&\ddots&\\&&\lambda_l}$. 
For each $\lambda_i$ with multiplicity $l_i$, say $\lambda_i=\lambda_{i+1}=\cdots=\lambda_{i+l_i-1}$, consider the subextension $L/K(\lambda_i)$ with $\Gamma_i={\rm Gal}(L/K(\lambda_i))$. 
By a version of Hilbert's Satz 90 \cite[Chapter 10, Proposition 3]{SerreGTM67} 
asserting $H^1(\Gamma_i,{\rm GL}_{l_i}(L))=0$, 
the eigenspace $V_i\cong L^{l_i}$ of $\lambda_i$ has an $L$-basis 
$(\bm{v}_i, \bm{v}_{i+1},\ldots, \bm{v}_{i+l_i-1})$ 
whose elements are fixed by $\Gamma_i$. 

Now we have $V=(\bm{v}_j)_j\in {\rm GL}_l(L)$ with $S=VDV^{-1}$. 
Take an action $\sigma \in \Gamma\act \{1,2,\ldots, l\}$ such that $\sigma(\lambda_i)=\lambda_{\sigma(i)}$. 
Since conjugate eigenspaces have the same dimensions, we may assume that  
$\sigma\in \Gamma$ acts on $\bm{v}_j$'s via $\sigma(\bm{v}_j)=\bm{v}_{\sigma(j)}$, so $\sigma(V)=(\bm{v}_{\sigma(j)})_j=VP_\sigma$ with a permutation matrix $P_\sigma\in {\rm GL}_l(\Z)$. 

To have $\bm{u}^{\!\top}S\bm{v}=\bm{u}^{\!\top}VDV^{-1}\bm{v}=\sum_i b_i\lambda_i$, it suffices to write 
\[\bm{u}^{\!\top}V=(u_i')_i^{\!\top}, \ V^{-1}\bm{v}=(v_i')_i\] 
and put $b_i=u_i'v_i'$. 
This $\bm{b}=(b_i)_i$ satisfies the desired conditions. Indeed, 
the equality $\bm{u}^{\!\top}S^n\bm{v}=\sum_i b_i \lambda_i^n$ is obvious. 
In addition, by 
\begin{align*}
&\sigma((u_i')_i^{\!\top})=\sigma(\bm{u}^{\!\top}V)=\bm{u}^{\!\top}VP_\sigma=(u_{\sigma(i)}')_i^{\!\top},\\ 
&\sigma((v_i')_i)=\sigma(V^{-1}\bm{v})=P_\sigma^{-1}V^{-1}\bm{v}=P_\sigma^{\top}V^{-1}\bm{v}=(v_{\sigma(i)}')_i,
\end{align*}  
we obtain 
$\sigma(b_i)=\sigma(u_i'v_i')=u_{\sigma(i)}'v_{\sigma(i)}'=b_{\sigma(i)}$.  
\end{proof} 

\subsection{Main proof} 

\begin{proof}[Proof of {\rm \Cref{thm.def.AK} (1) $\Leftrightarrow$ (2)}] \ 

\subsubsection*{\ul{{\rm (1) $\Rightarrow$ (2)}}}    
Suppose that $\alpha\in \mca{A}_K$ is represented by a linear recurrent sequence $(a_n)_n$ over $K$ with a separable eigen polynomial of length $l+1$, say, 
$a_n=c_{1}a_{n-1}+\cdots +c_{l}a_{n-l}$ with $c_i\in K$, $c_{l}\neq 0$. 
Consider the companion matrix 
$C=\spmx{c_1&c_2&\cdots& c_{l-1}&c_l\\
1&0&\cdots &0&0\\
0&1&\cdots &0&0\\
\vdots&\vdots&\ddots&\vdots&\vdots\\
0&0&\cdots&1&0}$. 
Then we have 
$(a_{n+1-i})_i=C(a_{n-i})_i$ in $K^n$, and hence 
$a_n=\bm{u}^{\!\top}C^n\bm{v}$ for 
$\bm{u}=\bm{e}_1, \bm{v}=(a_{1-i})_i$ in $K^n$. 

By the assumption that the eigen polynomial of $C$ is the product of separable polynomials in $K[x]$, $C$ admits the Jordan--Chevalley decomposition (cf.\,\cite{CoutyEsterleZarouf2011-arXiv}), namely, 
we may write \[C=SU\]
with $S,U\in {\rm M}_l(K)$ such that $S$ is semisimple (i.e., diagonalizable), $U$ is unipotent, and $SU=US$. 

Let $P$ be a monic irreducible element in $R$ and put $d={\rm deg} P$.  
Let $d\gg 0$ so that $U^{q^d}=I_l$. Then we obtain 
\[a_{q^d}=\bm{u}^{\!\top}C^{q^d}\bm{v}
=\bm{u}^{\!\top}S^{q^d}U^{q^d}\bm{v}
=\bm{u}^{\!\top}S^{q^d}\bm{v}.\]  
Note that if in addition $d\gg0 $ so that $P$ is coprime to the denominator of every entry of $U$, $\bm{u}$, $\bm{v}$, then $a_{q^d}\mod P\in R/(P)$ makes sense. 

Again, by the assumption, the eigenvalues are in a finite Galois extension $L/K$, and $S$ is diagonalizable over $L$. 
By \Cref{lem.diag-perm}, there exists $(b_i)_i\in L^l$ such that 
\[a_{q^d}=\sum_i b_i \lambda_i^{q^d}\]
and $\Gamma={\rm Gal}(L/K)$ acts on the set $\{(b_i,\lambda_i)\mid i\}$ of pairs. 
Hence $\sum_i b_i\otimes \lambda_i\in (L\otimes L)^\Gamma$, 
and \Cref{lem.GinvAL} implies that $g:=\psi(\sum_i b_i\otimes \lambda_i)\in A(L)$. 

Suppose that $P$ is unramified in $L/K$ and coprime to the denominator of the image of $g$. 
Then, for every prime ideal $\mf{p}$ of $L$ over $(P)$, we have 
\[g(\varphi_\mf{p})=\sum_i b_i \varphi_\mf{p}(\lambda_i)\equiv \sum_i b_i \lambda_i^{q^d} \mod \mf{p}.\]
By \Cref{prop.Frobev.AK}, we may write $a_{q^d}\equiv g(\varphi_P)\mod P$. 
Thus we obtain $\alpha=[(a_{q^{{\rm deg} P}}\mod P)_P]\in {\rm Im}\,{\rm ev}_L$. 

\subsubsection*{\ul{{\rm (2) $\Rightarrow$ (1)}}} 
Let $L/K$ be a finite Galois extension with $\Gamma={\rm Gal}(L/K)$ and let $g\in A(L)$ with $\alpha=[(g(\varphi_P)\mod P)_P]\in \mca{A}_K$. 
By \Cref{lem.GinvAL}, $\psi^{-1}(g)\in (L\otimes L)^\Gamma$. 
By \Cref{lem.Ginv-perm}, we may write $\psi^{-1}(g)=\sum_i b_i\otimes \lambda_i$ 
with $b_i,\lambda_i\in L$ such that $\Gamma$ acts on the set $\{(b_i,\lambda_i)\mid i\}$ of pairs. 
Define a sequence $(a_n)$ by $a_n=\sum_i b_i \lambda_i^n$. 
Omit $i$'s with $\lambda_i=0$ and 
put $p(t)=\prod_i (t-\lambda_i)=t^l+s_1t^{l-1}+\cdots +s_l$. 
Then the coefficients are fixed under the $\Gamma$-action, so $p(t)\in K[t]$, 
and we obtain a recurrence formula 
\[a_n+s_1a_{n-1}+\cdots + s_la_{n-l}=0\]
over $K$. 
The calculation in (1) $\Rightarrow$ (2) assures that \[a_{q^d}\equiv g(\varphi_P) \mod P \]
for every $P$ with ${\rm deg} P\gg 0$. 
\end{proof} 

\section{Proof of \Cref{thm.def.AK} (2) $\Leftrightarrow$ (3)} \label{sec.proof.(2)iff(3)}

Let the setting be as before. 
Let $L/K$ be a finite Galois extension with $\Gamma={\rm Gal}(L/K)$ and let $P$ be a monic irreducible element of $R=R_K$ unramified in $L/K$. 
Let $F_{P,L}$ denote the $q^{{\rm deg} P}$-power map on $R_L/(P)R_L=\prod_{\mf{p}\mid (P)}R_L/\mf{p}$, 
which is an automorphism of $R/(P)$-algebras. 

\begin{dfn} \label{def.AKFrob}
\emph{The $\mca{A}_K$-valued Frobenius automorphism} $F_{\mca{A}_K,L}$ is the $\mca{A}_K$-valued automorphism of $\mca{A}_L=L\otimes \mca{A}_K$ 
induced by $(F_{P,L})_P$. 
\end{dfn}

\begin{lem} \label{lem.Kspan} 
Take a $K$-basis of $L$ and consider the presentation matrix of $F_{\mca{A}_K,L}$ in $\mca{M}_n(\mca{A}_K)$. 
Then the $K$-span of the matrix coefficients coincides with the image of the Frobenius-evaluation map ${\rm Im}\,{\rm ev}_L=\{[(g(\varphi_P)\mod P)_P]\in \mca{A}_K\mid g\in A(L)\}$. 
\end{lem} 

\begin{proof} The $K$-span of the matrix coefficients of $F_{\mca{A}_K,L}$ is the image of the map 
\[L^\vee\otimes L\to\mca{A}_K;\ 
f\otimes y\mapsto[(f(y^{q^{{\rm deg} P}}) \mod P)_P],\] 
where $L^\vee={\rm Hom}_K(L,K)$. 
Since the trace form ${\rm Tr}_{L/K}:L\times L\to K; (x,y)\mapsto\sum_{\sigma\in \Gamma}\sigma(xy)$ 
is non-degenerate, we have an isomorphism $L\congto L^\vee$, and the image of the above map coincides with that of 
\[L\otimes L\congto L^\vee\otimes L\to \mca{A}_K;\ x\otimes y\mapsto [(\sum_{\sigma\in \Gamma}\sigma(xy^{q^{{\rm deg} P}}) \mod P)_P].\]
 
On the other hand, 
by \Cref{lem.Ginv-perm}, we have a surjective homomorphism ${\rm Tr}_{L/K}:L\otimes L\surj (L\otimes L)^\Gamma;$ $\sum_i x_i\otimes y_i\mapsto \sum_{\sigma \in \Gamma}\sigma(x_i)\otimes \sigma(y_i)$. 
By \Cref{lem.GinvAL}, Bourbaki's canonical isomorphism $\psi:L\otimes L\congto {\rm Map}(\Gamma,L)$; $\sum_i x_i\otimes y_i\mapsto (\tau\mapsto \sum_i x_i\tau(y_i))$ restricts to a bijection $\psi:(L\otimes L)^\Gamma\congto A(L)$. 
By \Cref{prop.Frobev.AK}, we have a well-defined map ${\rm ev}_L:A(L)\to \mca{A}_K;$ $g\mapsto [(g(\varphi_\mf{p})\mod \mf{p})_\mf{p}] = [(g(\varphi_P)\mod P)_P]$. 
By taking the composition 
\[L\otimes L\underset{{\rm Tr}_{L/K}}{\surj} (L\otimes L)^\Gamma \underset{\psi}{\congto} A(L) \underset{{\rm ev}_L}{\to} \mca{A}_K;\]
\begin{align*}
&x\otimes y\mapsto \sum_{\sigma\in \Gamma} \sigma(x)\otimes \sigma(y)
\mapsto (\tau\mapsto \sum_{\sigma} \sigma(x)\tau(\sigma(y))\\ 
&\mapsto [(\sum_{\sigma} \sigma(x)\varphi_P(\sigma(y)) \mod P)_P]
=[(\sum_{\sigma} \sigma(xy^{q^{{\rm deg} P}})) \mod P)_P],
\end{align*} 
we see that ${\rm Im}\,{\rm ev}_L$ coincides with the $K$-span above. 
\end{proof}

\begin{proof}[Proof of {\rm \Cref{thm.def.AK} (2) $\Leftrightarrow$ (3)}] 
Done by \Cref{lem.Kspan}. 
\end{proof} 

\section{Proof of \Cref{thm.subset.AK}} \label{sec.thm.subset.AK} 

\subsection{$\mca{P}^0_{\mca{A}_K}$ is a $K$-subalgebra of $\mca{A}_K$} 

Recall that $\alpha\in \mca{A}_K$ satisfies $\alpha\in \mca{P}^0_{\mca{A}_K}$ iff $\alpha\in {\rm Im}\, {\rm ev}_L$ for some finite Galois extension $L/K$. 

\begin{lem} \label{lem.LL'} 
For finite Galois extensions $L,L'/K$ with $L\subset L'$, 
the natural injective homomorphism $A(L)\inj A(L')$ is compatible with the Frobenius evaluation maps. 
\[\xymatrix{ A(L) \ar^{{\rm ev}_L}[r] \ar@{^(->}[d]& \mca{A}_L \ar@{^(->}[d]\\ 
A(L') \ar^{{\rm ev}_{L'}}[r]& \mca{A}_{L'} 
}
\]
\end{lem}

\begin{proof} Let $\mf{p}$, $\mf{P}$ be unramified primes of $R_L$, $R_{L'}$ with $\mf{P}\mid \mf{p}$. Then by 
\begin{align*}
&A(L)\underset{{\rm ev}_L}{\to} \mca{A}_L \inj \mca{A}_{L'};\ g\mapsto (g(\varphi_\mf{p})\mod \mf{p})_{\mf{p}}
\mapsto (g(\varphi_\mf{p})\mod \mf{P})_{\mf{P}},\\ 
&A(L)\inj A(L')\overset{{\rm ev}_{L'}}{\to}  \mca{A}_{L'}; \ 
g\mapsto (g(\varphi_{\mf{P}}|_L) \mod \mf{P})_{\mf{P}}
=(g(\varphi_{\mf{p}})\mod \mf{P})_{\mf{P}},
\end{align*}
the diagram commutes. 
\end{proof}

\begin{proof}[Proof of {\rm \Cref{thm.subset.AK} (1)}] 
Let $\alpha,\beta\in \mca{P}_{\mca{A}_K}^0$. 
Then there are finite Galois extensions $L_1/K$, $L_2/K$ and elements $g\in A(L_1)$, $h\in A(L_2)$ satisfying ${\rm ev}_{L_1}(g)=\alpha$, ${\rm ev}_{L_2}(h)=\beta$. 
Take the composition field $L=L_1L_2$, which is again a finite Galois extension of $K$. 
By \Cref{lem.LL'}, we may regard $g,h\in A(L)$, so we obtain
$\alpha+\beta={\rm ev}_{L}(g+h)$, $\alpha\beta={\rm ev}_{L}(gh)\in \mca{P}_{\mca{A}_K}^0$.
In addition, let $k\in K$. Then by $kg\in A(L_1)$, we have $k\alpha={\rm ev}_{L_1}(kg)\in \mca{P}_{\mca{A}_K}^0$. 
Thus, $\mca{P}_{\mca{A}_K}^0$ is a $K$-subalgebra of $\mca{A}_K$. 
\end{proof}

\subsection{Proper inclusions $K\subsetneq \mca{P}^0_{\mca{A}_K}\subsetneq \mca{C}^{\rm sep}_{\mca{A}_K}\subsetneq \mca{C}^{\rm alg}_{\mca{A}_K}$}
 
\begin{lem} \label{lem.KP0} 
We have $K\subsetneq \mca{P}^0_{\mca{A}_K}$. 
\end{lem}

\begin{proof} Let $(a_P)_P\in \prod_P R/(P)$. 
If $\alpha=[(a_P)_P]$ is in the image of $K$, then 
there is a unique $a\in R$ such that,  
$a\mod P=a_P$ for every $P$ with ${\rm deg}P\gg 0$. 
Note that the Dirichlet--Kornblum--Landau theorem on arithmetic progressions and the quadratic reciprocity law (or the Chebotarev density theorem, see \Cref{lem.Cheb})   
yield that the Dirichlet densities of $P$'s with $(\frac{\theta}{P})=1$ and $-1$ are both $1/2$. 
Hence, the linear recurrent sequences $(F_n)_n$ in \Cref{eg.P0AK} gives nontrivial elements of $\mca{P}^0_{\mca{A}_K}$. 
\end{proof}

\begin{rem} The construction of \Cref{eg.P0AK} (2) was initially inspired from our another work \cite[Section 5]{SakamotoTangeUeki2026CMB} on knots and primes. 
We wonder if we can give a similar construction for $2\mid q$ cases by using the quadratic symbol in characteristic 2 \cite{ConradK2010QRchar2}. 
\end{rem}

\begin{lem} \label{lem.P0Csep} 
We have 

{\rm (1)} $\mca{P}^0_{\mca{A}_K}\subset \mca{C}^{\rm sep}_{\mca{A}_K}$, 

{\rm (2) (i)} $\mca{P}^0_{\mca{A}_K}$ is countable, 
{\rm (ii)} $\mca{C}^{\rm sep}_{\mca{A}_K}$ is uncountable.  

So, $\mca{P}^0_{\mca{A}_K}\subsetneq \mca{C}^{\rm sep}_{\mca{A}_K}$. 
\end{lem}

\begin{proof}
(1) 
Let $\alpha\in \mca{P}^0_{\mca{A}_K}$. By \Cref{thm.def.AK} (2), 
there is a finite Galois extension $L/K$ and a map $g\in A(L)$ with ${\rm ev}_L(g)=[(g(\varphi_P)\mod P)_P]=\alpha$. 
For each $\sigma\in {\rm Gal}(L/K)$, let $f_\sigma(x)\in K[x]$ denote the minimal polynomial of $g(\sigma)\in L$ over $K$. Put $f(x)=\prod_\sigma f_\sigma(x)$. 
Then $f(x)$ is a product of separable polynomials with $f(g)=0$, so $f(\alpha)=0$, 
and hence $\alpha\in \mca{C}^{\rm sep}_{\mca{A}_K}$.

(2) 
Note that $R$ is a countable set. 

(i) Since a linear recurrence formula over $K$ and the initial condition to define a linear recurrent sequence correspond to elements of $R^{\oplus \Z}$, 
linear recurrent sequences correspond to elements of $R^{\oplus \Z}\otimes R^{\oplus \Z}$. 
By \Cref{thm.def.AK} (1), there is a surjection $R^{\oplus \Z}\otimes R^{\oplus \Z} \surj P^0_{\mca{A}_K}$, hence $P^0_{\mca{A}_K}$ is countable. 

(ii) For each $S\subset \Z_{\geq 0}$, 
define 
$a_{q^{{\rm deg}P}}=\begin{cases} 0\ \text{if} \ {\rm deg}P\in S\\ 1\ \text{if} \ {\rm deg} P\not\in S\end{cases}\!\!\!\!\!\!$ 
and put $\alpha_S=[(a_{q^{{\rm deg} P}} \mod P)_P]\in \mca{A}_K$. 
Then $\alpha_S$ is a root of $x^2-x$, which is separable in any characteristic, so $\alpha_S\in \mca{C}^{\rm sep}_{\mca{A}_K}$. 
Since $2^{\Z_{\geq 0}}$ is uncountable and $\bigoplus_P R/(P)$ is countable, 
the set of $\alpha_S$'s is uncountable, and so is $\mca{C}^{\rm sep}_{\mca{A}_K}$. 
\end{proof}

\begin{lem} \label{lem.CsepCalg}
We have $\mca{C}^{\rm sep}_{\mca{A}_K}\subsetneq \mca{C}^{\rm alg}_{\mca{A}_K}$. 
\end{lem}
 
\begin{proof} Let $f(x)=x^q-\theta$. 
If $P\neq \theta$, then $0\neq \theta \mod (P)$ in $R/(P)=\F_{q^d}$, and $e={\rm ord}\,\theta$ in $\F_{q^d}^\times$ divides $q^d-1$.
Since the set $\mu_e$ of primitive $e$-th roots of unity in $\F_{q^d}$ admits an automorphism $x\mapsto x^q$, $f(x)=x^q-\theta \mod P$ always has a root in $\mca{A}_K$. 

Let $\alpha$ be a root of $f(x)=x^q-\theta$ in $\mca{A}_K$. 
Suppose that there exists a separable polynomial $g(x)\in K[x]$ with $g(\alpha)=0$. 
Since $f(x)$ is irreducible and non-separable while $g(x)$ is separable, 
there exist $u,v\in K[x]$ with $uf+vg=1$. 
This contradicts $f(\alpha)=g(\alpha)=0$. Hence $\alpha$ is non-separable and $\alpha\in \mca{C}^{\rm alg}_{\mca{A}_K}\setminus \mca{C}^{\rm sep}_{\mca{A}_K}$. 
\end{proof} 

\begin{rem} The proper inclusion $\Q\subsetneq \mca{P}^0_{\mca{A}}\subsetneq \mca{C}_{\mca{A}}$ for Rosen's ring may be proved similarly. 
\end{rem}

\begin{rem} 
Consider the map $\Delta: R^{\Z_{\geq 0}}\to \mca{A}_K;$ $(a_{q^d})_d\mapsto [(a_{q^{{\rm deg} P}})_P]$. 
For each subset $\mca{S}\subset \mca{A}_K$, 
let $\mca{S}^\Delta=\mca{S}\cap {\rm Im}\,\Delta$. 

By \Cref{thm.def.AK} (1), we have $\mca{P}^{0\Delta}_{\mca{A}_K}=\mca{P}^{0}_{\mca{A}_K}$. 

By the proof of \Cref{lem.P0Csep} (2) (ii), $\mca{C}^{\rm sep \Delta}_{\mca{A}_K}$ is also uncountable, and hence $\mca{P}^0_{\mca{A}_K}\subsetneq \mca{C}^{\rm sep \Delta}_{\mca{A}_K}$. 
If we replace $S\subset \Z_{\geq 0}$ by a set $S$ of monic irreducible elements, then we obtain more elements, so we have $\mca{C}^{\rm sep \Delta}_{\mca{A}_K} \subsetneq \mca{C}^{\rm sep }_{\mca{A}_K}$. 
Explicit constructions of $\alpha \in \mca{C}^{\rm sep \Delta}_{\mca{A}_K}\setminus \mca{P}^0_{\mca{A}_K}$
would be of further interest. 

The proof of \Cref{lem.CsepCalg} further yields an element that assures $\mca{C}^{\rm sep \Delta}_{\mca{A}_K} \subsetneq \mca{C}^{\rm alg \Delta}_{\mca{A}_K}$ by choosing a common element of $R$ for each degree. 
\end{rem}

\subsection{``Transcendental'' elements} 
\subsubsection{In $\mca{A}$} \label{ss.transA} 
We first recall the cases over $\Q$. 
According to Anzawa--Funakura \cite{AnzawaFunakura2024},
an \emph{$\mca{A}$-transcendental number} stands for an element of the complement $\mca{A}\setminus \mca{C}_\mca{A}$. 
Their criterion may be rephrased as 
\begin{prop}[(cf.~{\cite[Proposition 3.7]{AnzawaFunakura2024})}] \label{prop.trans.A} 
Let $(a_p)_p\in \prod_p \Z/p\Z$. 
If $\alpha=[(a_p)_p]\in \mca{C}_\mca{A}$, then there are only finitely many $a\in \Z$ such that ``$a_p=a\mod p$ for infinitely many $p$'s''. 
\end{prop} 

\begin{eg} 
[(``transcendental'' elements of $\mca{A}$, in several senses)] \label{eg.trans} \ 

\renewcommand{\labelitemi}{$\bullet$}
\begin{itemize} 
\item Anzawa--Funakura \cite[Example 3.8]{AnzawaFunakura2024}. 
If $(a_p)_p=(1,1,2,1,2,3,1,2,3,4,\ldots)$, then \Cref{prop.trans.A} assures that 
$\alpha=[(\alpha_p)_p]\in \mca{A}\setminus \mca{C}_\mca{A}$ . 

\item 
Anzawa--Funakura \cite[Theorem 1.2]{AnzawaFunakura2024}.  
Under the generalized Riemann hypothesis (GRH), $q$-Fibonacci sequences define elements in $\mca{A}\setminus \mca{C}_\mca{A}$. 

\item Luca--Zudilin \cite{LucaZudilin2025Ramanujan}.  
(Without assuming GRH,) 
$q$-Fibonacci sequences define elements in $\mca{A}\setminus \mca{P}^0_\mca{A}$.

\item Luca--Zudilin \cite{LucaZudilin2025-arXiv} The Frobenius traces of an elliptic curve define an element in $\mca{A}\setminus \mca{P}^0_\mca{A}$. 
  
\item[$\circ$] 
Non-zero finite MZVs are generically expected to lie in $\mca{A}\setminus \mca{C}_\mca{A}$ and $\mca{A}\setminus \mca{P}_\mca{A}$, as Kaneko--Zagier's conjecture suggests, yet no example has been shown to be even non-zero, which remains a notoriously difficult open problem (cf.\, Kaneko \cite{Kaneko2019PMB}, Seki \cite{Seki2024Integers}). 
  
\item[$\circ$] 
Kaneko--Matsusaka--Seki 
\cite{KanekoMatsusakaSeki2025IMRN} introduces and compares several finite analogues of Euler's constant 
$\ds \gamma=\lim_{n\to \infty}(1+\frac{\,1\,}{2}+\frac{\,1\,}{3}+\cdots+\frac{\,1\,}{n}-\log n)$ 
from the perspectives of a ``regularized value of $\zeta(1)$'', 
that are expected to be in $\mca{A}\setminus \mca{P}_\mca{A}$ as suggested by Kontsevich--Zagier's work \cite{KontsevichZagier2001} (see \cite[Remark 4.2]{KanekoMatsusakaSeki2025IMRN}). 
The relationship between $\mca{C}_\mca{A}$ and $\mca{P}_\mca{A}$ is likely still unknown today. 

\item[$\circ$] Matsusaka--Miyazaki--Yara \cite{MatsusakaMiyazakiYara2026IJNT} 
study a finite analogue of Dobi\'nski's formula related to the Napier constant $e$, 
as well as partially extends the previous work on Euler's constant. 

\item[$\bullet$] Matsusaka--Seki \cite{MatsusakaSeki2026-arXiv-NaiveTrans} 
thoroughly updates the works of Anzawa--Funakura and Luca--Zudilin,
as well as points out and suggests further examples of elements in $\mca{A}\setminus \mca{C}_\mca{A}$. Among other things, they removed GRH in \cite[Theorem 1.2]{AnzawaFunakura2024}. 

\item[$\bullet$] Mihara \cite{Mihara2026-arXiv-APNSA} 
extends Matsusaka--Seki's transcendence criteria by applying non-standard analysis and model theory. 
\end{itemize}
\end{eg}

\subsubsection{In $\mca{A}_K$} 
\begin{prop} \label{prop.trans} 
Let $(a_P)_P\in \prod_P R/(P)$. 
If $\alpha=[(a_P)_P]\in \mca{C}^{\rm alg}_{\mca{A}_K}$, 
then there are only finitely many $x\in R$ such that 
``$a_P=x\mod P$ for infinitely many $P$'s''. 
\end{prop} 

\begin{proof} Let $f(x)\in R[x]$ with $f(\alpha)=0$. 
Then $f(a_P) = 0$ in $R/(P)$ for almost all $P$'s. 
Suppose that $a\in R$ satisfies $a_P=a\mod P$ for infinitely many $P$'s. 
Then $f(a)\equiv 0\mod P$ for infinitely many $P$'s. 
Since $f(a)\in R$ has only finitely many divisors, we must have $f(a)=0$ in $R$, 
and there are only finitely many such $a$'s. 
\end{proof} 

\begin{eg} \label{eg.trans.AK} 
Let $(a_{q^n})_n=(1,1,\theta,1,\theta,\theta^2,1,\theta,\theta^2,\theta^3,\ldots)$. 
Then, by \Cref{prop.trans}, $\alpha=[(a_{q^{{\rm deg}P}})_P]\in \mca{A}_K\setminus \mca{C}^{\rm alg}_{\mca{A}_K}$. 
\end{eg} 

\begin{rem} Exploring ``transcendental'' elements in $\mca{A}_K$ of venerable backgrounds, such as those in \Cref{eg.trans}, would be of great further interest.  
\end{rem} 

\begin{proof}[Proof of {\rm \Cref{thm.subset.AK} (2)}] 
Lemmas \ref{lem.KP0}, \ref{lem.P0Csep}, \ref{lem.CsepCalg}, and \Cref{eg.trans.AK} complete the proof.  
\end{proof} 

\section{Proofs of \Cref{cor.Frob.AK} and \Cref{thm.root.AK}} \label{sec.proofs.Frob/root.AK} 
Let us derive two consequences of \Cref{thm.def.AK} (1) $\Leftrightarrow$ (2). 

\begin{proof}[Proof of {\rm \Cref{cor.Frob.AK}}] 
Let $S$ be a set of monic irreducible elements of $R$. 

Suppose that there is a linear recurrent sequence $(a_n)_n$ over $K$ with separable eigen polynomial such that 
$S$ is a cofinite subset of $\{P\mid a_{q^{{\rm deg} P}}\equiv 0\mod P\}$.
By \Cref{thm.def.AK} (1) $\Leftrightarrow$ (2), there is a finite Galois extension $L/K$ and $g\in A(L)$ such that $[(g(\varphi_P)\mod P)_P]= [(a_{q^{{\rm deg}P}}\mod P)_P]$, so we have $g(\varphi_P)\equiv 0 \mod P$ iff $a_{q^{{\rm deg}P}}\equiv0 \mod P$ for almost all $P$'s. 
Let $C_g=\{\sigma\in {\rm Gal}(L/K)\mid g(\sigma)=0\}$. 
Since $g(\sigma\tau\sigma^{-1})=\sigma g(\tau)$ for every $\sigma,\tau\in \Gamma$, the set $C_g$ is a union of conjugacy classes of ${\rm Gal}(L/K)$.
Since $S_{L,C_g}$ and $S$ almost coincide, $S$ is Frobenian. 

Conversely, suppose that $S$ is Frobenian, so there are $L/K$ and $C$ such that $S_{L,C}$ and $S$ almost coincide. 
Let $g:{\rm Gal}(L/K)\to L$ denote the characteristic function of $C$. Then $g\in A(L)$. 
Again, by \Cref{thm.def.AK} (1) $\Leftrightarrow$ (2), there is a linear recurrent sequence $(a_n)_n$ over $K$ with separable eigen polynomial 
such that $a_{q^{{\rm deg}P}} \equiv g(\varphi_P)\mod P$ for almost all $P$'s.
Hence, $\{P\mid a_{q^{{\rm deg}P}} \equiv 0 \mod P\}$ and $\{P\mid \varphi_P\subset C\}$ almost coincide. 
By multiplying a constant to $(a_n)_n$, we can modify the zero set up to any finite set, 
except for the set $\{P\mid a_P=0 \text{\ in\ }K\}$, 
to have that $S$ is a cofinite subset of $\{P\mid a_{q^{{\rm deg} P}}\equiv 0\mod P\}$. 
\end{proof}

\begin{rem} \label{rem.S={}}
Both in Corollaries \ref{cor.Rosen.1.3} and \ref{cor.Frob.AK}, it remains an interesting question to ask 
whether we always have a sequence $(a_n)_n$ with the equality 
$S=\{p\mid a_p\equiv 0\mod p\}$ or $S=\{P\mid a_P\equiv 0\mod P\}$ 
for a given Frobenian set $S$. 
A construction of $(a_n)_n$ with a detection of exceptional primes in \cite[Theorems 1.3, 1.6]{RosenTakeyamaTasakaYamamoto2024JNT} under a simple setting would give a clue. 
\end{rem}

\begin{proof}[Proof of {\rm \Cref{thm.root.AK}}] 
Let $\alpha\in \mca{P}^0_{\mca{A}_K}$ and $f(x)\in K[x]$ with $f(\alpha)=0$. 
Take $(a_P)_P\in \prod_P R/(P)$ with $\alpha=[(a_P)_P]$. 
By definition (\Cref{thm.def.AK} (2)), 
we have a finite Galois extension $L/K$ with $\Gamma={\rm Gal}(L/K)$ and $g\in A(L)$ such that 
$a_P=g(\varphi_P)\mod P$ for almost all $P$'s, so 
$f(a_P)=f(g(\varphi_P))\mod P=0$ for almost all $P$'s.  
Let $e\in \Gamma$ denote the identity element. 
By the Chebotarev density theorem (\Cref{lem.Cheb}), the Dirichlet density of unramified primes $\mf{p}$ of $L$ 
with $\varphi_\mf{p}\in \{e\}$  is $1/\#\Gamma$, 
so we have $f(g(e))\mod P=0$ for infinitely many $P$'s, and hence $f(g(e))=0$. 
By $g\in A(L)$, we have $g(e)=\sigma(g(e))$ for every $\sigma\in \Gamma$, 
so $g(e)\in K$. 
This completes the proof. 
\end{proof} 

\section{Proofs of \Cref{thm.density.AK}} \label{sec.proof.density.AK} 

\subsection{Preliminaries} 

\subsubsection{Classics}
We invoke the following general results. 

\begin{lem}[(Chebotarev's density theorem,  cf.\,{\cite[Theorem 9.13 A, B]{Rosen2002GTM}})]
\label{che} \label{lem.Cheb} 
Let $K$ be a function field with the constant field $\F_q$ (e.g., $K=\F_q(\theta)$ as before)  and 
let $\mca{S}_K$ denote the set of all primes (non-zero prime ideals) $P$ of $K$.  
Define the Dirichlet density of each $\mca{M}\subset \mca{S}_K$ by 
\[\delta(\mca{M})=\lim_{s\to 1+0} \frac{\sum_{P\in \mca{M}}(q^{{\rm deg}P})^{-s}}{\sum_{P\in \mca{S}_K}(q^{{\rm deg}P})^{-s}}.\]
Let $L/K$ be a finite Galois extension with $\Gamma={\rm Gal}(L/K)$ and let $C\subset \Gamma$ be a conjugacy class.
Put $\mca{S}'_K=\{P\in \mca{S}_K\mid \text{unramified in\ } L/K\}$. 
Let $\varphi_P$ denote the Frobenius conjugacy class of $P$, that is, the set of all $\varphi_\mf{p}$'s with $\mf{p}\mid P$.  
Let $S'_{L,C}=\{P\in \mca{S}'_K\mid \varphi_P=C\}$. 
Then 
\[\delta(S'_{L,C})=\frac{\#C}{\#\Gamma}.\]
If $L/K$ is a geometric extension (i.e., $K$ and $L$ have the same constant field), then the natural density 
\[\lim_{x\to \infty} \frac{\#\{P\in S'_{L,C}\mid {\rm deg}P\leq x\}}{\#\{P\in \mca{S}'_K\mid {\rm deg}P\leq x\}}\]
exists and coincides with the Dirichlet density. 
\end{lem} 
\begin{lem}[({\cite[Chapter VI, Corollary 7.4]{Neukirch}, see also \cite{RosenM1987ExpMath}})] \label{hil} \label{lem.principal} 
Let $K=\F_q(\theta)$ and let $L/K$ be a finite extension. 
A non-zero prime ideal $\mf{p}$ of $R_L$ completely splits in the Hilbert class field (i.e., the maximal unramified abelian extension) of $L$ if and only if $\mf{p}$  is a principal ideal of $R_L$. 
\end{lem}

\subsubsection{Rosen's lemmas} 
Let the setting be as in \Cref{ss.P0AK}, that is, $R=\F_q[\theta]$, $K=\F_q(\theta)$, etc., and let $L/K$ be a finite Galois extension with $\Gamma={\rm Gal}(L/K)$. 
Let $f(x)\in K[x]$ be a monic polynomial that completely decomposes as $f(x)=\prod_i(x-\alpha_i)$ in $L[x]$, so $f(x)$ is a product of separable polynomials. 
Put $\Gamma_i={\rm Gal}(L/K(\alpha_i)) \subset \Gamma$ for each $i$. 

Rosen's argument over $\mca{A}$ completely applies to the case over $\mca{A}_K$ to yield the following. 

\begin{lem}[(cf.\,{\cite[Lemma 3.1]{Rosen2020JNT}})] \label{lem1} \label{lem.Rosen.1}
Let $P$ be a prime of $K$ that is unramified in $L$ and 
coprime to the denominators of the coefficients of $f(x)$. 
Then $f(x)\mod P$ has a root in $R/(P)$ iff the Frobenius conjugacy class $\varphi_P\subset \Gamma$ is contained in $S_1=\bigcup_i \Gamma_i$. 
\end{lem}

\begin{lem}[(cf.\,{\cite[Lemma 3.2]{Rosen2020JNT}})] \label{lem2} \label{lem.Rosen.2}
Define 
$S_2=\bigcup_i\{\sigma\in\Gamma \mid C_{\Gamma}(\sigma)\subset\Gamma_i\}$,
where $C_\Gamma(\sigma)$ denotes the centralizer of $\sigma$ in $\Gamma$. 
Then, for every $g\in A(L)$, we have 
$\{\sigma\in \Gamma\mid f(g(\sigma))=0\}\subset S_2$, 
and there exists $g\in A(L)$ for which the equality of sets holds. 
\end{lem}

The following is purely group theoretic, where we replace $\Gamma'$ in \cite[Lemma 3.3]{Rosen2020JNT} with a general $S'$. 

\begin{lem}[(cf.\,{\cite[Lemma 3.3]{Rosen2020JNT}})] \label{lem3} \label{lem.Rosen.3} 
Let $\Gamma$ be a finite group and $A$ a finite abelian group. 
Consider the wreath product 
$\Gamma':=A^{\#\Gamma}\rtimes\Gamma$ 
with a natural surjection $\pi:\Gamma'\surj \Gamma;$ $(f,\sigma)\mapsto\sigma$. 
Let $S\subset \Gamma$ be a subset and put $S'=\pi^{-1}(S)\subset \Gamma'$.  
Then there are at least \[(1-\frac{\#\Gamma}{\#A})\#S'\] elements $\xi=(f,\sigma)\in S'$ that satisfy  
\[\pi(C_{\Gamma'}(\xi))\subset\langle\sigma\rangle.\]  
\end{lem}

\begin{proof} 
Let $\xi=(f,\sigma)\in \Gamma'=A^{\#\Gamma}\rtimes \Gamma$. 
If 
\[
\pi(C_{\Gamma'}(\xi))\not\subset \langle \sigma\rangle_{\Gamma}, \tag{$\star$}
\]
then there exists some $\tau\in \pi(C_{\Gamma'}(\xi))$ such that $\tau\not\in\langle \sigma\rangle_{\Gamma}$. 
By the former half of the proof of \cite[Lemma 3.3]{Rosen2020JNT} using additive Hilbert's Satz 90, 
for each $\sigma,\tau\in \Gamma$, there are at most $\#A^{\#\Gamma-1}$ elements $f\in A^{\#\Gamma}$ such that $\xi=(f,\sigma)$ satisfies ($\star$). 
Thus the number of $\xi=(f,\sigma)\in S'$ satisfying ($\star$) is at most 
\[\#\Gamma \ccdot \#S \ccdot \#A^{\#\Gamma-1}
=\frac{\#\Gamma}{\#A} \ccdot \#S\ccdot \#A^{\#\Gamma}
=\frac{\#\Gamma}{\#A}\#S'.\] 
Hence there are at least $(1-\frac{\#\Gamma}{\#A})\#S'$ elements 
$\xi=(f,\sigma)\in S'$ satisfying $\pi(C_{\Gamma'}(\xi))\subset \langle \sigma\rangle_{\Gamma}$. 
\end{proof} 

\subsubsection{Wreath product extensions} 
The Kummer/Artin--Schreier theory proves the following. 
\begin{lem} \label{lem.wreathext} 
Let $K=\F_q(\theta)$ as before and let $L/K$ be a finite Galois extension with $\Gamma={\rm Gal}(L/K)$. 
Let $r\in \Z_{>0}$. 
Then there is a finite Galois extension $L'/K$ with $L\subset L'$ such that 
${\rm Gal}(L'/K)\cong A^{\#\Gamma} \rtimes \Gamma$ with $A=(\Z/2\Z)^r$.
\end{lem}

\begin{proof} 
Suppose $2\nmid q$. 
By Lemmas \ref{lem.Cheb}, \ref{lem.principal}, we may take distinct primes $P_1,\ldots, P_r$ of $K$ that completely split in the Hilbert class field of $L$. 
Let $\beta_i\in R_L$ be a generator of a prime of $L$ over each $P_i$. 
Then the tuple $(\sigma(\beta_i))_{i,\sigma}$ indexed by $1\leq i\leq r$ and $\sigma\in \Gamma$ is independent in the multiplicative group $L^\times/(L^\times)^2$, 
since otherwise there is a subset $\{w_1,\ldots, w_k\}\subset \{\sigma(\beta_i) \mid i,\sigma\}$ with $\prod_i w_i\in (L^\times)^2$, so we have $\vartheta\in L^\times$ with $\prod_i w_i=\vartheta^2$, 
but the $w_1$-adic valuations of both sides tell that this never occurs, noting that $w_i$'s are coprime.  
Hence by the Kummer theory, if we put $L'=L(\{\sqrt{\sigma(\beta_i)}\mid \sigma\in \Gamma, 1\leq i\leq r\})$, then $L'/L$ is a Galois extension with $\Gamma':={\rm Gal}(L'/L)\cong (\Z/2\Z)^{rd}$, $d:=[L:K]$.  

Consider the short exact sequence 
\begin{equation}
1\to A^{\#\Gamma} \to \Gamma'\underset{\pi}{\to} \Gamma\to 1. \tag{$\star$}
\end{equation}
If $\tau\in \Gamma$ and $\tau'\in \pi^{-1}(\tau)$, 
then we see that $\tau'(\sqrt{\sigma(\beta_i)})=\pm \sqrt{\tau\sigma(\beta_i)}$. 
Since $\#\pi^{-1}(\tau)=[L':L]=2^{rd}$, 
there is a unique $\tau'$ such that 
$\tau'(\sqrt{\sigma(\beta_i)})=\sqrt{\tau\sigma(\beta_i)}$ for every $(\sigma,i)$. 
Define a map $s:\Gamma\to \Gamma'$ by sending $\tau$ to such a unique $\tau'$. 
This $s$ is a group homomorphism; For each $\tau_1,\tau_2\in \Gamma$, 
we have $s(\tau_1\tau_2)(\sqrt{\sigma(\beta_i)})
=\sqrt{\tau_1\tau_2\sigma(\beta_i)}
=s(\tau_1)(\sqrt{\tau_2\sigma(\beta_i)})
=s(\tau_1)s(\tau_2)(\sqrt{\sigma(\beta_i)})$ for every $(\sigma,i)$, 
so $s(\tau_1\tau_2)=s(\tau_1)s(\tau_2)$. 
Thus ($\star$) admits a splitting, and hence $\Gamma'\cong A^{\#\Gamma}\rtimes \Gamma$. 

Suppose instead $2\mid q$. 
Put $h(x)=x^2+x$. Then the quotient $L/h(L)$ of the additive group is an infinite group. 
Indeed, take any element $\vartheta$ of $R_L$ generating a principal prime ideal 
and put $\beta_i=\frac{1}{\vartheta^{2i-1}}$ for $1\leq i\leq r$. 
Let $\varpi$ denote the valuation with $\varpi(\vartheta)=1$. 
If $x\in L$ satisfies $\varpi(x)=m\in \Z_{<0}$, then $\varpi(x^2+x)=2m$. 
So, if $0<i<j$, then $\beta_j-\beta_i \not\in h(L)$. 
Thus, the tuple $(\sigma(\beta_i))_{\sigma,i}$ is independent in $L/h(L)$. 
Choose a root $\gamma_{\sigma,i}$ of $x^2+x-\sigma(\beta_i)$ for each $(\sigma,i)$. 
Then, by the Artin--Schreier theory (cf.\,\cite[Theorem 15]{McCarthy1991book}), if we put $L'=L(\{\gamma_{\sigma,i}\mid \sigma,i\})$, 
then $L'/L$ is a Galois extension with ${\rm Gal}(L'/L)\cong (\Z/2\Z)^{rd}$, $d=[L:K]$. A similar argument assures  $\Gamma'\cong A^{\#\Gamma}\rtimes \Gamma$. 
\end{proof}

\subsection{Main proof} 

\begin{proof}[Proof of {\rm \Cref{thm.density.AK}}]
Let $f(x)\in K[x]$ be a product of separable polynomials. 
Let $(a_P)_P\in \prod_P R/(P)$ with $[(a_P)_P]\in \mca{P}^0_{\mca{A}_K}$. Then we clearly have 
\begin{equation}
\delta(\{P\mid f(a_P)=0\}) \leq 
\delta(\{P\mid f\text{\ has a root in\ }R/(P)\}). \tag{$\ast$}
\end{equation}
Let $L/K$ and $g\in A(L)$ be as in \Cref{thm.def.AK} (2), that is, we have 
$a_P=g(\varphi_P)\mod P$ for almost all $P$'s. 
Suppose in addition $f(x)=\prod_i(x-\alpha_i)$ in $L[x]$, 
and put $\Gamma_i={\rm Gal}(L/K(\alpha_i))$, 
$S_1=\bigcup_i \Gamma_i$, $S_2=\bigcup_i\{\sigma\in\Gamma \mid C_{\Gamma}(\sigma)\subset\Gamma_i\}$ as before, 
so we have $S_2\subset S_1\subset \Gamma$.   

By Lemmas \ref{lem.Rosen.1}, \ref{lem.Rosen.2} and the Chebotarev density theorem (\Cref{lem.Cheb}), 
we have
\begin{gather*}
\delta(\{P\mid f\text{\ has a root in\ }R/(P)\})=\frac{\#S_1}{\#\Gamma},\\
\max_{g\in A(L)}\delta(\{P\mid f(a_P)=0\})= \max_{g\in A(L)} \delta(\{P\mid f(g(\varphi_P))\mod P =0)=\frac{\#S_2}{\#\Gamma}. 
\end{gather*}
Note that the value $\frac{\#S_1}{\#\Gamma}$ does not depend on the choice of $L$, 
while $\frac{\#S_2}{\#\Gamma}$ does.

If $f(x)$ has no root in $K$, then we have $\Gamma_i\subsetneq \Gamma$ for every $i$, and $C_\Gamma(e)=\Gamma \not\subset \Gamma_i$, so $e\not\in S_2$. By $e\in S_1\setminus S_2$, we obtain $S_2\subsetneq S_1$. 
Thus, the inequality $(\ast)$ is strict. Hence no element of $\mca{P}^0_{\mca{A}_K}$ realizes the supremum.  

Let $r\in \Z_{>0}$. Then by \Cref{lem.wreathext}, there exists a Galois extension $L'/K$ with $L\subset L'$ such that $\Gamma':={\rm Gal}(L'/K) \cong (\Z/2\Z)^{r\#\Gamma}\rtimes \Gamma$ 
with the projection $\pi:\Gamma'\surj \Gamma;$ $(f,\sigma)\mapsto \sigma$. 
Put $\Gamma_i'=\pi^{-1}(\Gamma_i)={\rm Gal}(L'/K(\alpha_i))\subset \Gamma'$, $S_1'=\bigcup_i \Gamma_i'=\pi^{-1}(S_1)$, and 
$S_2'=\bigcup_i \{\sigma\in \Gamma'\mid C_{\Gamma'}(\sigma)\subset \Gamma_i'\}=\pi^{-1}(S_2)$. 
Then by \Cref{lem.Rosen.3}, there are at least 
$(1-\frac{\#\Gamma}{2^r})\#S_1'$ elements $\sigma\in S_1'$ satisfying $\pi(C_{\Gamma'}(\sigma))\subset \langle \pi(\sigma)\rangle_\Gamma$. 
Thus we have 
\[(1-\frac{\#\Gamma}{2^r})\#S_1'\leq \#S_2' \leq \#S_1',\]
hence $(1-\frac{\#\Gamma}{2^r})\frac{\#S_1'}{\#\Gamma'}\leq \frac{\#S_2'}{\#\Gamma'} \leq \frac{\#S_1'}{\#\Gamma'}$. 
If we take $r\to \infty$, then we have $(1-\frac{\#\Gamma}{2^r})\to 1$,  
while $\frac{\# S_1'}{\#\Gamma'}=\frac{\# S_1}{\#\Gamma}$ does not depend on $r$, so we obtain $\frac{\# S_2'}{\#\Gamma'}\to \frac{\# S_1}{\#\Gamma}$.  
This completes the proof. 
\end{proof}

\section{Artin $t$-motives and $K^{\rm sep}$} \label{sec.t-motive} 

Recall Remarks \ref{rem.PA0.motive} and \ref{rem.t-motive}. 
Rosen's ring $\mca{P}_\mca{A}^0$ is regarded as a finite analogue of $\ol{\Q}$ from the viewpoint of the periods of motives (algebraic varieties). 
Here, we discuss periods in positive characteristic to partially support a similar interpretation in our case. 

The notion of \emph{$t$-motives} was introduced by Anderson \cite{Anderson1986Duke} in connection with $t$-modules, which are analogues of abelian varieties with complex multiplication. 
Taelman \cite{Taelman2009JNT} introduced \emph{Artin t-motives} as a subclass of analytically trivial $t$-motives, thereby highlighting the Galois-theoretic structure within the Tannakian framework established for analytically trivial $t$-motives.
The history of \emph{periods} of $t$-motives traces back to Carlitz's work in the 1930s, whereas the modern framework is due to many people, including Papanikolas \cite{Papanikolas2008Invent}. 
For more on $t$-motives and related topics, see the proceedings volume \cite{BockleGossHartlPapanikolas2020}. 

Let $K=\F_q(\theta)$ as before.  
(In what follows, this $K$ may be replaced by its separable extension.)  
Define a map $\tau$ on $K^{\rm sep}(t)=K^{\rm sep}\otimes_{\F_q}\F_q(t)$ by $a\otimes f\mapsto a^q\otimes f$.  
Let $M$ be an \emph{effective $t$-motive} over $K$ in the sense of Taelman, 
namely, $M$ is a free and finitely generated $K[t]$-module endowed with an endomorphism $\sigma:M\to M$ satisfying 
$\sigma(fm)=\tau(f)\sigma(m)$ for all $f\in K[t]$ and $m\in M$, such that,  
${\rm det}\,\sigma$ is a power of $t-\theta$ up to multiplication by units in $K$. 
Let $\mf{m}=(m_j)_j$ be a $K[t]$-basis of $M$ and let $\Phi\in {\rm M}_r(K[t])\cap\, {\rm GL}_r(K(t))$ denote the matrix presenting the action of $\sigma$ on $\mf{m}$ in the sense that, for each $\bm{a}=(a_i)_i\in K[t]^r$, we have $\sigma:\mf{m} \bm{a}\mapsto \mf{m}\Phi \tau(\bm{a})$ 
with $\tau(\bm{a})=(\tau(a_i))_i$. 

\begin{eg} (1) The trivial $t$-motive $\mathbf{1}$ is the $K[t]$-module $K[t]$
with $\sigma(a)=\tau(a)$ 
and  $\Phi=(1)$. 

(2) \emph{The Carlitz motive} $C$ is the $K[t]$-module $K[t]e$ with 
$\sigma(ae)=\tau(a)(t-\theta)e$ and $\Phi=(t-\theta)$. 
\end{eg}

Let $\C_\infty=\wh{\ol{K_\infty}}=\wh{\ol{\F_q(\!(1/\theta)\!)}}$ be a complete and algebraically closed field as in \Cref{rem.t-motive} 
and let $\C_\infty\{t\}$ denote the Tate algebra, 
that is, the ring of formal power series whose coefficients converge to 0. 
Let $M$ be an effective $t$-motive with $\sigma,$ $\mf{m},$ and $\Phi$. 
If there is $\Psi\in {\rm GL}_r({\rm Frac}\,\C_\infty\{t\})$ such that $\tau(\Psi^{\top})=\Psi^{\top}\Phi$, then $M$ is said to be \emph{rigid analytically trivial} 
and $\Psi$ is called a \emph{rigid analytic trivialization} of $(\mf{m},\Phi)$. 
If an entry of $\tau(\Psi^{-1})$ converges at $t=\theta$, then the limit value is called a \emph{period} of $M$. 
In fact, the limit matrix $\tau(\Psi^{-1})_{t=\theta}$ may be interpreted as a presentation matrix of the comparison isomorphism between the de Rham cohomology and the Betti cohomology of $M$.  

An effective $t$-motive $M$ is called an \emph{Artin $t$-motive} if $M\otimes_K K^{\rm sep}\cong r\bm{1}\otimes_K K^{\rm sep}$ with $r\in \Z_{\geq 1}$, where $r\bm{1}$ denotes the effective $t$-motive whose $\sigma$-action on some basis is presented by $I_r\in {\rm M}_r(K[t]) \cap {\rm GL}_r(K(t))$. 
If an effective $t$-motive $M$ with $(\mf{m},\Phi)$ admits a rigid analytic trivialization $\Psi \in {\rm GL}_r(K^{\rm sep}[t])$, then $M$ is an Artin $t$-motive. Indeed, 
if we put $\mf{w}=\mf{m}(\Psi^{\top})^{-1}$, then we have 
$\mf{w}\Psi^{\top}=\mf{m}$, 
$\mf{m}\Phi=\sigma(\mf{m})
=\sigma(\mf{w}\Psi^{\top})
=\sigma(\mf{w})\tau(\Psi^{\top})
=\sigma(\mf{w})\Psi^{\top}\Phi$, and hence 
$\sigma(\mf{w})=\mf{m}(\Psi^{\top})^{-1}=\mf{w}$.
This implies that $M$ is an Artin $t$-motive, by \cite[Proposition 4.1.3]{Taelman2009JNT}. 

\begin{prop} \label{prop.ksep}
The set of all periods of all Artin $t$-motives over $K$ 
generates $K^{\rm sep}$. 
\end{prop}

\begin{proof} 
Let $M$ be an Artin $t$-motive over $K$ with a $K[t]$-basis $\mf{m}$. 
Then, by definition, 
there is some $Z\in {\rm GL}_r(K^{\rm sep}[t])$ such that the matrix presenting the action of $\sigma$ on the $K[t]$-basis $\mf{m}Z$ is $I_r$. 
Hence, for any $\bm{a}=(a_i)_i \in K[t]^r$, we have
 \[\mf{m}Z I_r \tau(\bm{a})=\sigma(\mf{m}Z\bm{a})=\mf{m}\Phi \tau(Z)\tau(\bm{a}),\] 
so $Z=\Phi\, \tau(Z)$, and hence $\tau(Z^{-1})=Z^{-1}\Phi$.
Thus  $\Psi=(Z^{-1})^{\top} \in {\rm GL}_r(K^{\rm sep}[t])$ is a rigid analytic trivialization. If an entry of $\tau(\Psi^{-1})$ converges at $t=\theta$, then the limit value (i.e., a period) is in $K^{\rm sep}$. 
If $\Psi_1$ and $\Psi_2$ are rigid analytic trivializations, 
then $\Psi_1\Psi_2^{-1}\in {\rm GL}_r(\F_q(t))$. 
Thus, all periods of Artin $t$-motives are in $K^{\rm sep}$. 

To prove the converse, let $L/K$ be a finite Galois extension with ${\rm Gal}(L/K)=\{g_1,g_2,\ldots,g_r\}$ and let $\mf{m}=(m_i)_i\in L^r$ be a $K$-basis of $L$. 
Define $\Phi\in {\rm M}_r(K)$ by 
$\tau (\mf{m})=(\tau(m_i))_i=(m_i^q)_i=\mf{m}\Phi$. 
Now let $M=L\otimes_K K[t]$ denote the effective $t$-motive 
such that $\mf{m}$ is a $K[t]$-basis and the action of $\sigma$ on $\mf{m}$ is presented by $\Phi$.
%
Define $\Psi=(g_j(\mf{m}))_j=(g_i(m_j))_{ij} \in {\rm M}_r(L)$. 
For each $i$, we have $\tau(g_i(\mf{m}))=g_i(\tau(\mf{m}))=g_i(\mf{m}\Phi)=
g_i(\mf{m})\Phi$, so $\tau(\Psi)=\Psi\Phi$. 
Since $\Psi^{\top}\Psi$ presents the trace form of $L/K$ with respect to the basis $\mf{m}$, the non-degeneracy of the trace form yields that $\Psi^{\top}\Psi$ is invertible. 
So, $\Psi$ is invertible, and $\Psi^{\top} \in {\rm GL}_r(L)\subset {\rm GL}_r(K^{\rm sep}[t])$ is a rigid analytic trivialization of $(\mf{m},\Phi)$. 
In addition, by $\tau(\Psi)=\Psi\Phi$, we have $\Phi\in {\rm GL}_r(L)$.   
Thus, $M$ is an Artin $t$-motive.  

Now let $a\in K^{\rm sep}$.
If we choose $L/K$ and $\mf{m}$ with $a\in L$ and $m_1=a$, then $a$ is an entry of $\Psi=\tau(\Psi)\Phi^{-1}$. 
Therefore, $a$ is in the extension of $K$ generated by the entries of $\tau(\Psi^{-1})=\tau(\Psi^{-1})_{t=\theta}\in {\rm GL}_r(L)$, which are periods of the Artin $t$-motive $M$. 
\end{proof}

\section{Further remarks}\label{sec.remarks}

\begin{rem} 
As we may expect from the example of the Fibonacci sequence (\Cref{eg.Fn}), 
Rosen--Takeyama--Tasaka--Yamamoto \cite{RosenTakeyamaTasakaYamamoto2024JNT} proved that we may encode the information of prime decompositions in a given extension $L/\Q$ into an element of $\mca{P}^0_\mca{A}$, 
which extends to the scope of the Langlands reciprocity between Galois representations and modular forms. 
%
Given Lafforgue's complete proof of the Langlands correspondence for ${\rm GL}_n$ over function fields around 2000, one may reasonably expect greater progress in $\mca{P}^0_{\mca{A}_K}$. 
\end{rem}


\begin{rem} \label{rem.def(2)}
As suggested from a viewpoint of anabelian geometry, in the condition (2) of the definitions of $\mca{P}^0_\mca{A}$ and $\mca{P}^0_{\mca{A}_K}$ (Theorems \ref{thm.def.A}, \ref{thm.def.AK}), the target $L$ of a map $g \in A(L)$ may be replaced by 
the abelianization ${\rm Gal}(\overline{L}/L)^{\rm ab}$ of the absolute Galois group,
using the conjugate action ${\rm Gal}(L/\Q)\act {\rm Gal}(\overline{L}/L)^{\rm ab}$ and 
the Artin reciprocity isomorphism ${\rm Gal}(\overline{L}/L)^{\rm ab}\congto C_L=I_L/P_L$ to the idele class group. 
\end{rem}

\begin{rem} \label{rem.AT}
We find that the studies of $\mca{P}^0_\mca{A}$ and $\mca{P}^0_{\mca{A}_K}$ come close to the heart of the analogy between knots and primes in arithmetic topology (cf.\,\cite{Morishita2024}), especially in light of the ramification theories and the Chebotarev density theorems. 
For a 3-manifold $M$ endowed with a Chebotarev link $\mca{L}$ (cf.\,\cite{Mazur2025Po, McMullen2013CM, Ueki7}), 
an analogue of the absolute Galois group is defined by 
${\rm Gal}(M,\mca{L})=\varprojlim_{h\in {\rm Cov}(M,\mca{L})}{\rm Gal}\,h
=\varprojlim_{J\subset \mca{L}}\widehat{\pi}_1(M-J)$, 
where $h$ runs through finite branched covers of $M$ branched along a finite sublink of $\mca{L}$, 
$J$ runs through finite sublinks of $\mca{L}$, and 
$\wh{\pi}_1$ denotes the profinite completion of the fundamental group (cf.\,\cite{GropperUekiWang-NU}). 
We have 
a natural isomorphism 
${\rm Gal}(M,\mca{L})^{\rm ab}\congto C_{M,\mca{L}}=I_{M,\mca{L}}/P_{M,\mca{L}}$ to the idele class group (cf.\,\cite{Niibo2014, NiiboUeki2019TAMS}). 
Thus we may start from the condition (2) of the definitions of  $\mca{P}^0_\mca{A}$ and $\mca{P}^0_{\mca{A}_K}$. 
Among other things, finding correct analogues of periods of motives for knots and 3-manifolds in this context would be an important problem. 
\end{rem} 

\bibliographystyle{amsalpha} 
\bibliography{

\newcommand{\etalchar}[1]{$^{#1}$}
\providecommand{\bysame}{\leavevmode\hbox to3em{\hrulefill}\thinspace}
\providecommand{\MR}{\relax\ifhmode\unskip\space\fi MR }
\providecommand{\MRhref}[2]{%
  \href{http://www.ams.org/mathscinet-getitem?mr=#1}{#2}
}
\providecommand{\href}[2]{#2}
\begin{thebibliography}{RTTY24}

\bibitem[AF24]{AnzawaFunakura2024}
Takumi Anzawa and Hidetaka Funakura, \emph{Congruences for the
  {$q$}-{F}ibonacci sequence related to its transcendence}, Ramanujan J.
  \textbf{63} (2024), no.~4, 1057--1072. \MR{4721156}

\bibitem[And86]{Anderson1986Duke}
Greg~W. Anderson, \emph{{$t$}-motives}, Duke Math. J. \textbf{53} (1986),
  no.~2, 457--502. \MR{850546}

\bibitem[Ax68]{Ax1968AnnMath}
James Ax, \emph{The elementary theory of finite fields}, Ann. of Math. (2)
  \textbf{88} (1968), 239--271. \MR{229613}

\bibitem[BGHP20]{BockleGossHartlPapanikolas2020}
Gebhard B\"ockle, David Goss, Urs Hartl, and Matthew Papanikolas (eds.),
  \emph{{$t$}-motives: {H}odge structures, transcendence and other motivic
  aspects}, EMS Series of Congress Reports, EMS Publishing House, Berlin,
  [2020] \copyright 2020. \MR{4321963}

\bibitem[Bou03]{Bourbaki.AlbegraII}
Nicolas Bourbaki, \emph{Algebra {II}. {C}hapters 4--7}, english ed., Elements
  of Mathematics (Berlin), Springer-Verlag, Berlin, 2003. \MR{1994218}

\bibitem[CCM22]{ChangChenMishiba2022Camb}
Chieh-Yu Chang, Yen-Tsung Chen, and Yoshinori Mishiba, \emph{Algebra structure
  of multiple zeta values in positive characteristic}, Camb. J. Math.
  \textbf{10} (2022), no.~4, 743--783. \MR{4524827}

\bibitem[CCM23]{ChangChenMishiba2023Pi}
\bysame, \emph{On {T}hakur's basis conjecture for multiple zeta values in
  positive characteristic}, Forum Math. Pi \textbf{11} (2023), Paper No. e26,
  32. \MR{4655529}

\bibitem[CEZ11]{CoutyEsterleZarouf2011-arXiv}
Danielle Couty, Jean Esterle, and Rachid Zarouf, \emph{D{\'e}composition
  effective de {Jordan}-{Chevalley} et ses retomb{\'e}es en enseignement},
  preprint. arXiv:1103.5020, 2011.

\bibitem[CM17]{ChangMishiba2017Bordeaux}
Chieh-Yu Chang and Yoshinori Mishiba, \emph{On finite {C}arlitz multiple
  polylogarithms}, J. Th\'eor. Nombres Bordeaux \textbf{29} (2017), no.~3,
  1049--1058. \MR{3745259}

\bibitem[Con10]{ConradK2010QRchar2}
Keith Conrad, \emph{Quadratic reciprocity in characteristic 2}, unpublished
  notes, \url{http://www.math.uconn.edu/
  kconrad/blurbs/ugradnumthy/QRchar2.pdf}, 2010.

\bibitem[GUW26]{GropperUekiWang-NU}
Nadav Gropper, Jun Ueki, and Yi~Wang, \emph{A {N}eukirch--{U}chida theorem for
  3-manifolds}, preprint. arXiv:2604.09469, April 2026.

\bibitem[Har22]{Harada2022JNT}
Ryotaro Harada, \emph{On multi-poly-{B}ernoulli-{C}arlitz numbers}, J. Number
  Theory \textbf{232} (2022), 406--422. \MR{4343836}

\bibitem[HK25]{HoriKida2025}
Haruto Hori and Masanari Kida, \emph{Linear recurrent sequences providing
  decomposition law in number fields}, Ramanujan J. \textbf{66} (2025), no.~3,
  Paper No. 58, 17. \MR{4861242}

\bibitem[IKL{\etalchar{+}}24]{ImKimLeNgoDacPham2024}
Bo-Hae Im, Hojin Kim, Khac~Nhuan Le, Tuan Ngo~Dac, and Lan~Huong Pham,
  \emph{Zagier-{H}offman's conjectures in positive characteristic}, Forum Math.
  Pi \textbf{12} (2024), Paper No. e18, 49. \MR{4812189}

\bibitem[Kan19]{Kaneko2019PMB}
Masanobu Kaneko, \emph{An introduction to classical and finite multiple zeta
  values}, Publications math\'ematiques de {B}esan\c con. {A}lg\`ebre et
  th\'eorie des nombres. 2019/1, Publ. Math. Besan\c con Alg\`ebre Th\'eorie
  Nr., vol. 2019/1, Presses Univ. Franche-Comt\'e, Besan\c con, [2019]
  \copyright 2019, pp.~103--129. \MR{4395010}

\bibitem[KMS25]{KanekoMatsusakaSeki2025IMRN}
Masanobu Kaneko, Toshiki Matsusaka, and Shin-ichiro Seki, \emph{On finite
  analogues of {E}uler's constant}, Int. Math. Res. Not. IMRN (2025), no.~2,
  Paper No. rnae281, 12. \MR{4849924}

\bibitem[Kon09]{Kontsevich2009JJM}
Maxim Kontsevich, \emph{Holonomic $\mathscr{D}$-modules and positive
  characteristic}, Jpn. J. Math. \textbf{4} (2009), no.~1, 1--25. \MR{2491280}

\bibitem[KZ01]{KontsevichZagier2001}
Maxim Kontsevich and Don Zagier, \emph{Periods}, Mathematics unlimited---2001
  and beyond, Springer, Berlin, 2001, pp.~771--808. \MR{1852188}

\bibitem[KZ26]{KanekoZagier2026}
Masanobu Kaneko and Don Zagier, \emph{Finite multiple zeta values}, to appear
  in Proceedings of the 17th {MSJ-SI} conference on Modular forms and Multiple
  Zeta values, Adv. Stud. Pure Math., 2026.

\bibitem[Lec53]{Lech1953AM}
Christer Lech, \emph{A note on recurring series}, Ark. Mat. \textbf{2} (1953),
  417--421. \MR{56634}

\bibitem[LZ25a]{LucaZudilin2025Ramanujan}
Florian Luca and Wadim Zudilin, \emph{Irrationality and transcendence questions
  in the `poor man's ad\`ele ring'}, Ramanujan J. \textbf{67} (2025), no.~4,
  Paper No. 88, 10. \MR{4921812}

\bibitem[LZ25b]{LucaZudilin2025-arXiv}
\bysame, \emph{Poor man's transcendence for {Frobenius} traces of elliptic
  curves}, Preprint, {arXiv}:2507.14773 [math.{NT}] (2025), 2025.

\bibitem[Mah35]{Mahler1935ETKF}
K.~Mahler, \emph{Eine arithmetische {Eigenschaft} der {Taylor}-{Koeffizienten}
  rationaler {Funktionen}.}, Proc. Akad. Wet. Amsterdam \textbf{38} (1935),
  50--60 (German).

\bibitem[Mat22]{Matsuzuki2022JIS}
Daichi Matsuzuki, \emph{Alternating variants of multiple poly-{B}ernoulli
  numbers and finite multiple zeta values in characteristic 0 and {$p$}}, J.
  Integer Seq. \textbf{25} (2022), no.~5, Art. 22.5.7, 21. \MR{4448013}

\bibitem[Maz25]{Mazur2025Po}
Barry Mazur, \emph{Primes, knots and {P}o}, Essays on topology---dedicated to
  {V}alentin {P}o\'enaru, Springer, Cham, 2025, pp.~13--31. \MR{4944626}

\bibitem[McC91]{McCarthy1991book}
Paul~J. McCarthy, \emph{Algebraic extensions of fields}, second ed., Dover
  Publications, Inc., New York, 1991. \MR{1105534}

\bibitem[McM13]{McMullen2013CM}
Curtis~T. McMullen, \emph{Knots which behave like the prime numbers}, Compos.
  Math. \textbf{149} (2013), no.~8, 1235--1244. \MR{3103063}

\bibitem[Mih26]{Mihara2026-arXiv-APNSA}
Tomoki Mihara, \emph{Notes on algebraic properties and non-standard analysis of
  the ring of integers modulo infinitely large primes}, preprint.
  arXiv:2605.00512, May 2026.

\bibitem[MMY26]{MatsusakaMiyazakiYara2026IJNT}
Toshiki Matsusaka, Taichi Miyazaki, and Shunta Yara, \emph{On finite analogues
  of {D}obi\'nski's formula and of {E}uler's constant via {G}regory
  polynomials}, to appear in Int. J. Number Theory.

\bibitem[Mor24]{Morishita2024}
Masanori Morishita, \emph{Knots and primes}, Universitext, Springer Singapore,
  2024, An introduction to arithmetic topology, 2nd edition.

\bibitem[MS26]{MatsusakaSeki2026-arXiv-NaiveTrans}
Toshiki Matsusaka and Shin-ichiro Seki, \emph{Some results on naive
  transcendence in the ring of integers modulo infinitely large primes},
  preprint. arXiv:2604.25566, April 2026.

\bibitem[MSW24]{MaesakaSekiWatanabe2024-arXiv}
Takumi Maesaka, Shin-ichiro Seki, and Taiki Watanabe, \emph{Deriving two
  dualities simultaneously from a family of identities for multiple harmonic
  sums}, Preprint, {arXiv}:2402.05730 [math.{NT}] (2024), 2024.

\bibitem[Neu99]{Neukirch}
J{\"u}rgen Neukirch, \emph{Algebraic number theory}, Grundlehren der
  Mathematischen Wissenschaften [Fundamental Principles of Mathematical
  Sciences], vol. 322, Springer-Verlag, Berlin, 1999, Translated from the 1992
  German original and with a note by Norbert Schappacher, With a foreword by G.
  Harder. \MR{1697859 (2000m:11104)}

\bibitem[Nii14]{Niibo2014}
Hirofumi Niibo, \emph{Id\`elic class field theory for 3-manifolds}, Kyushu J.
  Math. \textbf{68} (2014), no.~2, 421--436. \MR{3243372}

\bibitem[NU19]{NiiboUeki2019TAMS}
Hirofumi Niibo and Jun Ueki, \emph{Id\`elic class field theory for 3-manifolds
  and very admissible links}, Trans. Amer. Math. Soc. \textbf{371} (2019),
  no.~12, 8467--8488. \MR{3955553}

\bibitem[Pap08]{Papanikolas2008Invent}
Matthew~A. Papanikolas, \emph{Tannakian duality for {A}nderson-{D}rinfeld
  motives and algebraic independence of {C}arlitz logarithms}, Invent. Math.
  \textbf{171} (2008), no.~1, 123--174. \MR{2358057}

\bibitem[Ros87]{RosenM1987ExpMath}
Michael Rosen, \emph{The {H}ilbert class field in function fields}, Exposition.
  Math. \textbf{5} (1987), no.~4, 365--378. \MR{917350}

\bibitem[Ros02]{Rosen2002GTM}
\bysame, \emph{Number theory in function fields}, Graduate Texts in
  Mathematics, vol. 210, Springer-Verlag, New York, 2002. \MR{1876657}

\bibitem[Ros17]{Rosen2017-arXiv.choice-free}
Julian Rosen, \emph{A choice-free absolute {G}alois group and {A}rtin motives},
  preprint. arXiv:1706.06573, 2017.

\bibitem[Ros20]{Rosen2020JNT}
\bysame, \emph{A finite analogue of the ring of algebraic numbers}, J. Number
  Theory \textbf{208} (2020), 59--71. \MR{4032288}

\bibitem[RTTY24]{RosenTakeyamaTasakaYamamoto2024JNT}
Julian Rosen, Yoshihiro Takeyama, Koji Tasaka, and Shuji Yamamoto, \emph{The
  ring of finite algebraic numbers and its application to the law of
  decomposition of primes}, J. Number Theory \textbf{263} (2024), 335--365.
  \MR{4755038}

\bibitem[Sek24]{Seki2024Integers}
Shin-ichiro Seki, \emph{Regular primes, non-{W}ieferich primes, and finite
  multiple zeta values of level {$N$}}, Integers \textbf{24} (2024), Paper No.
  A22, 14. \MR{4710423}

\bibitem[Ser79]{SerreGTM67}
Jean-Pierre Serre, \emph{Local fields}, Graduate Texts in Mathematics, vol.~67,
  Springer-Verlag, New York-Berlin, 1979, Translated from the French by Marvin
  Jay Greenberg. \MR{554237 (82e:12016)}

\bibitem[Ser12]{Serre2012NXp}
\bysame, \emph{Lectures on {$N_X (p)$}}, Chapman \& Hall/CRC Research Notes in
  Mathematics, vol.~11, CRC Press, Boca Raton, FL, 2012. \MR{2920749}

\bibitem[Sko34]{Skolem1934EGDG}
Th. Skolem, \emph{Ein {Verfahren} zur {Behandlung} gewisser exponentialer
  {Gleichungen} und diophantischer {Gleichungen}.}, 8. {Skand}. {Mat}.
  {Kongr}., {Stockholm}, 1934, 163-188 (1934)., 1934.

\bibitem[STU26]{SakamotoTangeUeki2026CMB}
Honami Sakamoto, Ryoto Tange, and Jun Ueki, \emph{Liminal {${\rm
  SL}_2$}$\mathbb{Z}_p$-representations and odd-th cyclic covers of genus one
  two-bridge knots}, Canad. Math. Bull. (2026).

\bibitem[Tae09]{Taelman2009JNT}
Lenny Taelman, \emph{Artin {$t$}-motifs}, J. Number Theory \textbf{129} (2009),
  no.~1, 142--157. \MR{2468475}

\bibitem[Tha04]{Thakur2004}
Dinesh~S. Thakur, \emph{Function field arithmetic}, World Scientific Publishing
  Co., Inc., River Edge, NJ, 2004. \MR{2091265}

\bibitem[Uek21]{Ueki7}
Jun Ueki, \emph{Chebotarev links are stably generic}, Bull. Lond. Math. Soc.
  \textbf{53} (2021), no.~1, 82--91. \MR{4224512}

\end{thebibliography}
MatsuzukiSakamotoUeki.PosCharFAN.bbl}

\end{document}